\documentclass[12pt]{amsart}
\usepackage{amsmath,amsfonts,amssymb,latexsym}
\usepackage{hyperref}
\usepackage{graphicx}
\usepackage{comment}
\usepackage[a4paper, left=3cm, right=3cm, top=3cm, bottom=3cm]{geometry}
\usepackage[utf8]{inputenc}

\newcommand{\NN}{\mathbb{N}}
\newcommand{\CC}{\mathbb{C}}
\newcommand{\RR}{\mathbb{R}}
\newcommand{\QQ}{\mathbb{Q}}
\newcommand{\ZZ}{\mathbb{Z}}
\newcommand{\EE}{\mathbb{E}}
\newcommand{\Oo}{\mathcal{O}}
\newcommand{\Bo}{\mathcal{B}}

\newtheorem{Thm}{Theorem}
\newtheorem{Lem}{Lemma}
\newtheorem{Prop}{Proposition}

\theoremstyle{definition}
\newtheorem{Rem}{Remark}
\newtheorem{Def}{Definition}

\begin{document}
\title[Singular perturbations problems in moment differential equations]{Summability of formal solutions of some singular perturbations problems in differential and moment differential equations}
\author{Maria Ksi\c{a}\.zkiewicz}
\address{{Faculty of Mathematics and Natural Sciences,
College of Science,
Cardinal Stefan Wyszy{\'n}ski University,
W\'oycickiego 1/3,
01-938 Warszawa, Poland}\newline
ORCiD: 0009-0006-1147-2064}
\email{m.ksiazkiewicz@uksw.edu.pl}
\author{S{\l}awomir Michalik}
\address{{Faculty of Mathematics and Natural Sciences,
College of Science,
Cardinal Stefan Wyszy{\'n}ski University,
W\'oycickiego 1/3,
01-938 Warszawa, Poland}\newline
{\& Shibaura Institute of Technology,
Department of Engineering and Design,
Saitama 337-8570, Japan}\newline
ORCiD: 0000-0003-4045-9548}
\email{s.michalik@uksw.edu.pl}

\begin{abstract}
In this paper we study the summability of solutions of some general forms of singularly perturbed linear ordinary differential and moment differential equations. We conclude that under some assumptions solutions of these equations are summable. The type of this summability depends on the specific equation. We also show the connection between some singularly perturbed moment ordinary differential equations and some linear moment partial differential equations. 
We apply this connection to describe summable and multisummable formal solutions of these singularly perturbed moment ordinary differential equations.

Main techniques used to show these conclusions are based on Borel transforms, properties of solutions of moment partial differential equations and on the Cauchy integral formula together with integral representations of solutions of such equations.
\end{abstract}
\maketitle
\section{Introduction}
We can consider singular perturbations problems for both ordinary differential equations and their moment-differential generalisations. Many
authors write papers on this subject and consider different specific issues.

In some articles there are considered Gevrey solutions or Gevrey asymptotics, see for instance  M.~Canalis-Durand, J.-P.~Ramis, R.~Schäfke and Y.~Sibuya \cite{ramis} or Y.~Sibuya \cite{sibuya}.
Singular perturbations problems can also be investigated from the point of view of asymptotic integration, as in V. Sobchuk, I. Zelenska and V. Bobochko \cite{zelenska}. 
Some authors considered also systems of equations, for example Y. Sibuya \cite{sibuya1,sibuya2}. 

On the other hand, there also exist works on ordinary and partial moment differential equations, such as A. Lastra \cite{la3} and S. Michalik \cite{mi1,mi2}.

In this paper we consider summable and multisummable solutions of singular perturbations problems for some differential and moment differential equations. 
Similar problems of summability for some systems of differential equations have been investigated for example by by W. Balser and V.~Kostov \cite{bako}, by W. Balser and J. Mozo-Fernandez \cite{bamo} and by
M. Canalis-Durand, J. Mozo-Fernández and R. Schäfke \cite{canalis}.
 Moreover, the first attempt of study of multisummable solutions of singular perturbations problems for moment differential equations has been given in A. Lastra, S. Michalik and M. Suwińska \cite{lamisu}.
 
In the paper we get similar results for other classes of differential and moment differential equations.

Our starting point is the singular perturbation problem
\begin{equation}
 \label{eq:starting_point}
 \left(\varepsilon\frac{d}{dz}-a\right)x(\varepsilon,z)=f(z),
\end{equation}
which has been described in detail by W.~Balser \cite[Section 13.3]{ba1} and W.~Balser and J. Mozo-Fernández \cite[Introduction]{bamo}. They showed that the formal solution $\hat{x}(\varepsilon,z)=\sum_{n=0}^{\infty}x_n(z)\varepsilon^n$ of (\ref{eq:starting_point}) is $1$-summable in a given direction $d$ if and only if the inhomogeneity $f(z)$ is holomorphic in a complex neighbourhood of the origin, $f(z)$ admits holomorphic continuation into some sector $S_{d-\arg a}$ in the direction $d-\arg a$ and its continuation is of exponential growth at most $1$ at the infinity.

In the first part of the paper we get the similar characterisations (see Theorems \ref{th1} and \ref{thm:2}) of summable solutions $\hat{x}(\varepsilon,z)$ in terms of the inhomogeneity $f(z)$ for the following more general singular perturbations problems for differential and moment differential equations
\begin{equation*}
 P\left(\varepsilon^k \frac{d}{dz}\right)x(\varepsilon, z)=f(z)\qquad\text{and}\qquad P\left(\varepsilon^k \partial_{m,z}\right)x(\varepsilon, z)=f(z),
\end{equation*}
where $k\in\NN$, $P$ is a polynomial with complex coefficients such that $P(0)\neq 0$ and $m=(m(n))_{n\geq 0}$ is a regular sequence of moments. 

In the second part of the paper we discuss the similarity between the nature of the characterisation of $1$-summable formal solutions of the singular perturbation problem (\ref{eq:starting_point}) in terms of the inhomogeneity and the characterisation of $1$-summable formal solutions of the homogeneous heat equations in terms of the initial value, which has been already observed by W. Balser in \cite[Sections 13 and 13.4]{ba1}.
This connection is described in Theorems \ref{thm:3} and \ref{thm:3m},
which say that $\hat{x}(\varepsilon,z)$ is a solution of th singular perturbation problem for differential or moment differential equations
\begin{equation*}
    P\left(\varepsilon,z, \frac{d}{dz}\right) \hat{x}(\varepsilon,z) = \hat{f}(\varepsilon,z)\qquad\text{or}\qquad P\left(\varepsilon,z, \partial_{m,z}\right) \hat{x}(\varepsilon,z) = \hat{f}(\varepsilon,z)
\end{equation*}
if and only if $1$-Borel transform of $\hat{x}(\varepsilon,z)$ is a solution of the Cauchy problem for the appropriate partial or moment partial differential equation in two variables. 

Next, we use the above connection 
to develop the theory of formal solutions of singular perturbations problems for moments equations
\begin{equation}\label{eq:eq}
 P(\varepsilon,\partial_{m,z})x(\varepsilon,z)=f(\varepsilon,z),
\end{equation}
in the similar spirit to the theory of formal solutions of moment partial differential equations given in \cite{mi1,mi2,mi3}. In particular we decompose the formal solution $\hat{x}(\varepsilon,z)$ of (\ref{eq:eq}) on the sum of formal power series connected with the appropriate pseudodifferential operators (see Theorem \ref{thm:4}). We also apply this decomposition to get the sufficient conditions for summability and multisummability of the formal solution $\hat{x}(\varepsilon,z)$ of (\ref{eq:eq}) given in terms of the inhomogeneity $f(\varepsilon,z)$ (see Theorem \ref{thm:6}).

Main methods used to obtain described results
are given by Borel transforms, by Balser's theory of general moment summability \cite[Section 6.5]{ba1},  by properties of solutions of moment partial differential equations and by the Cauchy integral formula together with integral representations of solutions of such equations. We also use the theory of difference equations with constant coefficients.

This paper has the following structure. In the second section we introduce some basic definitions essential to understand this paper. Then we make a proof of a theorem on summability of some general form of singularly perturbed differential equations with a standard derivative. In the next section we make generalisation on the previous theorem for moment differential equations. In the next part of this article we discuss the relation of singularly perturbed differential equations to moment partial differential equations in two variables. In the next section we introduce moment pseudodifferential operators and formal solutions of singularly perturbed moment equations. In the last part of work we consider multisummable solutions of singularly perturbed equations.

\section{Preliminaries}
In this paper we will use standard notations: $\CC$ for a set of complex numbers, $\RR$ for a set of real numbers and $\NN$ for a set of natural numbers. We will also use the symbol $\NN_{0} = \NN \cup \{0\} = \{0, 1, 2, \dots \}$.  

Let $\EE$ be a Banach space with a norm $||\cdot||_{\EE}$. We will denote by $\mathbb{E}[[\varepsilon]]$ a space of all formal power series with coefficients from $\EE$: $$\mathbb{E}[[\varepsilon]] = \left\{\hat{x} :\ \hat{x}(\varepsilon) = \displaystyle \sum _{n=0} ^{\infty} x_{n} \varepsilon^{n},\ x_{n} \in \EE\ \textrm{ for every }n \in \NN_{0}\right\}.$$

Let $r \in \RR_{+}$. We denote by $D_{r}=\{z\in\CC: |z|<r\}$ an open disc with radius $r$. If its radius $r$ is not essential, we denote it shortly by $D$. We will use the symbol $\mathcal{O}(G,\EE)$ to denote a set of all $\EE$-valued holomorphic functions on the open set $G\subseteq \CC$. We also write $\Oo(G)$ if $\EE=\CC$ for simplicity.
Analogously, for given $\kappa\in\NN$, 
the set of all holomorphic functions of the variable $z^{1/\kappa}$
on $G$ (i.e. $1/\kappa$-holomorphic functions on $G$) is denoted by $\mathcal{O}_{1/\kappa}(G)$.

In this paper as a Banach space $\EE$ we will take the space $\CC$ of complex numbers with the standard norm $\|z\|_{\CC}=|z|$, or the space $\Oo$ of all holomorphic functions on $D_r$ and continuous on its closure $\overline{D}_r$ for some $r>0$, and equipped with a norm
\begin{equation}\label{eq:norm}
 \|\varphi\|_{\Oo}:=\sup_{|z|\leq r}|\varphi(z)|.   
\end{equation}
For given $\kappa\in\NN$ we will also use a more general Banach space $\Oo_{1/\kappa}$ of space of all $1/\kappa$-holomorphic functions on $D_r$, continuous on its closure $\overline{D}_r$ for fixed $r>0$, and equipped with the same norm (\ref{eq:norm}).


\begin{Def}
    Let $s \in \RR$ be fixed. We say that a formal power series $$\hat{x}(\varepsilon) = \displaystyle \sum _{n=0} ^{\infty} x_{n} \varepsilon^{n},\ \textrm{  where }\ x_{n} \in \EE\  \textrm{ for every }\ n \in \NN_{0}$$ has a \emph{Gevrey order} $s$ if there exist nonnegative constants $A,B < \infty$ such that for every $n \in \NN_{0}$ holds estimation: $$||x_{n}||_{\EE} \leq AB^{n}n!^{s}.$$ We will denote the space of all power series of Gevrey order $s$ by $\EE[[\varepsilon]]_{s}$.
\end{Def}

\begin{Def}A set $S(d, \alpha, R) \subset \mathbb{C}$ defined by $$S(d, \alpha, R):= \left\{t = re^{i \phi} \in \mathbb{C} : r \in (0, R), \phi \in \left(d-\frac{\alpha}{2}, d+\frac{\alpha}{2}\right)\right\}$$
is called a 
    \textit{sector in the direction $d \in \RR$ with opening $\alpha$ and radius $R \in (0,+\infty]$}.
    
    If a sector has an infinity radius, we use a symbol $S(d,\alpha)$ or $S_d$ if the opening $\alpha$ is not important
    By $\hat{S}_d$ we will denote a sum of a sector in the direction $d$ and a disc $D$, i.e. $\hat{S}_d=S_d\cup D$.
    We will also use the notation $\hat{S}_d(\alpha,r)$ to denote the set $S(d,\alpha)\cup D_r$.
\end{Def}

\begin{Def}
    We say that a function $f \in \mathcal{O}(\hat{S}_d,\EE)$ has \textit{exponential growth at most} $k\in\RR$ on a disc-sector $\hat{S}_d=\hat{S}_d(\alpha,r)$ if for every $\alpha'\in(0,\alpha)$ and $r'\in(0,r)$ there exist some constants $C,a < \infty$ for which the following estimate holds:
    \begin{equation*}
    \|f(z)\|_{\EE} \leq Ce^{a|z|^{k}}\quad \text{for all}\quad z \in \hat{S}_d(\alpha',r'). 
    \end{equation*}
    We denote the set of all such functions by $\mathcal{O}^{k}(\hat{S}_d,\EE)$.

    Analogously, we say that a function $f \in \mathcal{O}(\hat{S}_{d_1}\times\hat{S}_{d_2}))$ has \textit{exponential growth at most} $k_1,k_2\in\RR$ on a product of disc-sectors $\hat{S}_{d_1}=\hat{S}_{d_1}(\alpha_1,r_1)$ and $\hat{S}_{d_2}=\hat{S}_{d_2}(\alpha_2,r_2)$ if for every $\alpha_i'\in(0,\alpha_i)$ and $r_i'\in(0,r_i)$, $i=1,2$, there exist some constants $C,a_1,a_2 < \infty$ for which the following estimate holds:
    \begin{equation*}
    |f(\varepsilon,z)| \leq Ce^{a_1|\varepsilon|^{k_1}}e^{a_2|z|^{k_2}}\quad \text{for all}\quad (\varepsilon,z) \in \hat{S}_{d_1}(\alpha_1',r_1')\times\hat{S}_{d_2}(\alpha_2',r_2') . 
    \end{equation*}
    We denote the set of all such functions by $\mathcal{O}^{k_1,k_2}(\hat{S}_{d_1}\times\hat{S}_{d_2})$.
\end{Def}

\begin{Def}
    Let $\hat{x}(\varepsilon) = \displaystyle \sum _{n=0} ^{\infty} x_{n} \varepsilon^{n} \in \EE [[\varepsilon]]$ and $k > 0$. An operator $$(\hat{\mathcal{B}}_{k}\hat{x})(\varepsilon) = \displaystyle \sum _{n=0} ^{\infty} \frac{x_{n}\varepsilon^{n}}{\Gamma(1 + \frac{n}{k})}$$ is called \textit{a formal Borel operator with index} $k$.
\end{Def}

\begin{Def} \label{df:borel}
    A series $\hat{x}\in\EE[[\varepsilon]]$ is called $\textit{$k$-summable}$ in a direction $d$ if its formal Borel transform $\hat{\mathcal{B}}_{k}\hat{x}(\varepsilon)$ is convergent and its sum is holomorphic in a sector $\hat{S}_d$ with exponential growth at most $k$ as $\varepsilon\to\infty$. 
\end{Def}

\begin{Def}
    Let $d=(d_{1}, d_{2}.\dots, d_{n})\in \RR^{n}$ be a real vector and $k_{1}>k_{2}> \dots > k_{n}>0$. We call it an $\textit{admissible multidirection}$ with respect to $(k_{1}, k_{2}, \dots, k_{n})$ if $|d_{j}-d_{j-1}| \leq \frac{\pi}{2} \left(\frac{1}{k_{j}}-\frac{1}{k_{j-1}}\right)$ for $j=2,3,\dots, n$.
\end{Def}

\begin{Def}
    A series $\hat{x}(\varepsilon)$ is called $(k_{1}, \dots, k_{n})$-\textit{multisummable} in an admissible multidirection $d=(d_{1}, d_{2}, \dots, d_{n})\in \RR^{n}$ if $\hat{x}(\varepsilon) = \hat{x}_{1}(\varepsilon) + \dots + \hat{x}_{n}(\varepsilon)$ and for $i =1, \dots, n$ the series $\hat{x}_{i}(\varepsilon)$ is $k_{i}$-summable in the direction $d_{i}$. 
\end{Def}

\begin{Def}[see {\cite[pages 85--86, 89]{ba1}}]
    We call a pair of complex-valued functions $E(t), e(t)$ \textit{kernel functions of order $k>\frac{1}{2}$}  if:

    \begin{itemize}
        \item the function $e(t)$ is holomorphic in 
        $S(0, \frac{\pi}{k})$ and $t^{-1}e(t)$ is integrable at the origin, so for any $x_{0} > 0$ and $2k|\tau| < \pi$ there exists an integral $ \displaystyle \int _{0} ^{x_{0}} x^{-1} |e (ex^{i\tau})|dx$, for each $\varepsilon > 0$ there exist constants $c,K>0$ such that $|e(t)| \leq c \cdot \ \exp [-\left(\frac{t}{K}\right)^{k}]$ and $2k |\arg \ t| \leq \pi - \varepsilon$,
        \item if $x \in \mathbb{R}_{+}$, then $e(x) \in \mathbb{R}_{+}$,
        \item the function $E(t)$ is entire, 
        and has exponential growth at most $k$, it means that for some constants $M, N > 0$ we have estimation $|E(t)| \leq M e ^{N|t|^k}$ and in $S(\pi, \pi (2-\frac{1}{k}))$ the function $t^{-1}E\left(\frac{1}{t}\right)$ is integrable at the origin for any $t \in \CC$,
        \item the functions $e(t)$ i $E(t)$ are connected by \textit{a moment function} of order $\frac{1}{k}$: \\ $m(u) = \displaystyle \int ^{\infty} _0 x^{u-1} e(x) dx, \ Re \ u \geq 0$.\\ We assume that $m(0)=1.$\\ Using the moment function we can rewrite $E(t)$ in a form of power series: \\ $E(t) = \displaystyle \sum _{n=0} ^{+\infty} \frac{t^n}{m(n)}$.
    \end{itemize}
    Kernel functions can be defined also for $k \leq \frac{1}{2}$.
    Assume that there exists $p\in\mathbb{N}$ such that $pk>\frac{1}{2}$.
Then the functions $e,E$ are \emph{kernel functions of order $k$}, if there exist kernel functions $\tilde{e},\tilde{E}$ of order $pk$, such that
$e(t)=\displaystyle\frac{\tilde{e}(t^{1/p})}{p}$. Then we obtain:
$$E(t)=\displaystyle\sum_{n=0}^{\infty}\frac{t^n}{m(n)}=\displaystyle\sum_{n=0}^{\infty}\frac{t^n}
{\tilde{m}(pn)},$$ where $m(u),\tilde{m}(u)$ are moment functions connected respectively with pairs $e,E$ and $\tilde{e},\tilde{E}$ and there holds an equality
$m(u)=\tilde{m}(pu)$.\\
Every moment function of order $\frac{1}{k}$ determines a sequence of numbers $m=(m(n))_{n \in \NN_0}$ which is called a \emph{sequence of moments of order $\frac{1}{k}$}.
\end{Def}

\begin{Rem}
If $m=(m(n))_{n \in \NN_0}$ is a sequence of moments of order $\frac{1}{k}$, then there exist constants $a,A > 0$ that for every $n \in \mathbb{N}_0$ there holds an inequality
$$a^n n!^\frac{1}{k}\le m(n)\le A^n  n!^\frac{1}{k}.$$
\end{Rem}

\begin{Def}[{see \cite[Definition 2.4]{regular}}]
    A sequence of moments $(m(n))_{n \in \NN_0}$ of order $\frac{1}{k}$ is called  \textit{regular} if for every $n\in \NN_0$ there holds an estimation:
    $$ \ c n^{\frac{1}{k}} \leq \frac{m(n)}{m(n-1)}\leq C n^{\frac{1}{k}}, $$ where $c,C \in \mathbb{R}_{+}$ are constants.
\end{Def}

\begin{Def}[{see \cite[page 421]{bayo}}]
    Suppose $m = (m(n))_{n\in\NN_0}$ is a  fixed sequence of positive numbers such that $m(0)=1$.
    An operator $$\partial _{m,z} \left(\displaystyle \sum _{n=0} ^{\infty} \frac{a_{n}}{m(n)}z^{n}\right) = \displaystyle \sum _{n=0} ^{\infty} \frac{a_{n+1}}{m(n)}z^{n}$$
    is called a 
    $\textit{moment differential operator}$.
\end{Def}

\begin{Def}
    Let $m$ be a moment function of order $\frac{1}{k}$. An operator
    $$\hat{\mathcal{B}}_{m,z} \left(\displaystyle \sum _{n=0} ^{\infty} x_{n}z^{n}\right) = \displaystyle\sum^{\infty}_{n=0} \frac{x_n}{m(n)}z^{n}$$
    is called the
    \textit{formal $m$-Borel operator of order $k$}.
\end{Def}

\begin{Rem}
    For $m=(\Gamma(1+\frac{n}{k}))_{n\in\NN_0}$ we get a formal Borel operator $\hat{\Bo}_{k}$ with index~$k$.
\end{Rem}

\begin{Rem}
Let $m$ be a moment function of order $\frac{1}{k}$.
    By the general theory of moment summability (see  \cite[Section 6.5]{ba1}), we can replace in Definition \ref{df:borel} Borel transform $\hat{\mathcal{B}}_{k}$ by $m$-Borel transform $\hat{\mathcal{B}}_{m,z}$.
\end{Rem}


\section{Summability of some form of singularly perturbed differential equations}
In this chapter we present two lemmas and the theorem on the summability of singularly perturbed differential equations with a standard derivative.
\begin{Lem}\label{le:1}
    Suppose $\hat{x}(\varepsilon, z) \in \mathcal{O}[[\varepsilon]]$ is $k$-summable in the direction $d$ and $Q(\zeta)$ is a polynomial with complex coefficients. Then $Q(\varepsilon ^{k} \frac{d}{dz})  \hat{x}(\varepsilon, z)$ is also $k$-summable in the direction $d$.
\end{Lem}
\begin{proof}
    We can 
    write a polynomial $Q(\zeta)$ in the form $Q(\zeta)= a_{n}\zeta ^{n} + a_{n-1} \zeta ^{n-1} + \dots +a_{1} \zeta +a_{0}$. The rest of the proof of this lemma results from the fact that the space of $k$-summable formal power series forms an differential algebra, so it is stable under multiplication and derivation $\frac{d}{dz}$ (see \cite[Proposition 2]{pascal}).
\end{proof}

\begin{Lem}\label{le:2}
    Let $\hat{x}(\varepsilon, z) \in \mathcal{O}[[\varepsilon]]$ be a formal solution of the equation 
    $$\left(\varepsilon ^{k} \frac{d}{dz} - a\right)\hat{x}(\varepsilon, z) = f(z),
    $$
    where $a\in \CC \setminus \{0\}$, $k\in \NN$ and $f$ be a holomorphic function in a complex neighbourhood of zero. 
    Then a function $f$ can be continued analytically in some sector in the direction $kd-\arg a$ 
    and its continuation is of exponential growth of order $1$
    if and only if $\hat{x}(\varepsilon, z)$ is $k$-summable in the direction $d$.
\end{Lem}

\begin{proof}
    The proof of this lemma is well-known and can be found in \cite{bamo} on page 528.
    We have to rewrite $\hat{x}(\varepsilon, z)$ using the connection with $f(z)$.

    Let $\hat{x}(\varepsilon, z) = \displaystyle \sum _{n=0} ^{\infty} x_{n}(z)\varepsilon ^{n}$. Then the equation $\left(\varepsilon ^{k} \frac{d}{dz} - a\right)\hat{x}(\varepsilon, z) = f(z)$ can be written in the following form: $\displaystyle \sum_{n=k} ^{\infty} x_{n-k}'(z)\varepsilon^{n} - a\displaystyle \sum_{n=0} ^{\infty} x_{n}(z)\varepsilon^{n}  = f(z)$.

    We obtain the following formulas: $x_{0}= -\frac{f(z)}{a}$, $x_{n}(z)=0$ for $n \in \{1, 2, \dots, k-1 \}$ and $x_{n}(z)=\frac{1}{a}\partial _{z} x_{n-k}(z)$ for $n\geq k$.

Further, we can observe, that coefficients $x_{n}$ are nonzero if and only if $n$ is a multiplicity of $k$: $x_{pk}(z)=-\frac{f^{(p)}(z)}{a^{p+1}}$.

Now we can write the formal solution $\hat{x}(\varepsilon, z)$ of the equation in the following form:
$$\hat{x}(\varepsilon, z) = \displaystyle \sum _{n=0} ^{\infty} x_{n}(z)\varepsilon ^{n} = -\displaystyle \sum _{p=0} ^{\infty} \frac{f^{(p)}(z)}{a^{p+1}}\varepsilon ^{pk}.$$
At the end we can compute the Borel transform of $\hat{x}(z,\varepsilon)$:
$$\hat{\mathcal{B}}_{k}\hat{x}(\varepsilon, z)=-\displaystyle \sum _{p=0} ^{\infty} \frac{f^{(p)}(z)}{p!\cdot a^{p+1}}\varepsilon ^{pk} = -\frac{1}{a}\displaystyle \sum _{p=0} ^{\infty} \frac{f^{(p)}(z)}{p!}\cdot\left(\frac{\varepsilon^{k}}{a}\right)^{p} =-\frac{1}{a}f\left(z+\frac{\varepsilon ^{k}}{a}\right).$$
Therefore $\hat{\mathcal{B}}_{k}\hat{x}(\varepsilon, z)$ has the exponential growth at most $k$ on the set $D \times \hat{S}_{d}$ if and only if $f$ can be analytically continued in the direction $kd-\arg a$ and $f$ has the exponential growth at most $1$ on the set $\hat{S} _{kd-\arg a}$.
\end{proof}

\begin{Thm} \label{th1}
Let $k\in\NN$ and $d\in\RR$. W assume that $\hat{x}(\varepsilon, z) \in \EE[[\varepsilon]]$ is a formal solution of the equation
 \begin{equation}
 \label{eq:3}
 P\left(\varepsilon^k \frac{d}{dz}\right)x(\varepsilon, z)=f(z),
 \end{equation}
 where $P(\zeta)$ is any polynomial of degree $q$ with complex coefficients, such that $P(0) \neq 0$. We also assume that  $a_{1}, \dots , a_{l}\in\CC$ are the roots of the equation $P(\zeta)=0$  with multiplicity respectively  $p_{1}, \dots , p_{l}$, where $\sum _{i=1} ^{l} p_{i} = q$, $f(z)$ is holomorphic in a complex neighbourhood of the origin and $d\in\RR$.
 
 Then $f$ can be continued analytically in some sectors in directions $kd-\arg a_{i}$ for $i=1, \dots ,l$ and its continuation is of exponential growth of order $1$ if and only if $\hat{x}(\varepsilon,z)$ is $k$-summable in the direction $d$. 
\end{Thm}    

\begin{proof}
    $(\Longrightarrow)$ We can assume that 
\begin{equation*}
P(\zeta) = c_{q}\zeta^{q} + \dots +c_{1} \zeta + c_{0}\quad\text{for some}\quad c_{0},c_1, \dots, c_{q} \in \mathbb{C},\quad c_0, c_{q} \neq 0.
\end{equation*}

We can assume that $\hat{x}(\varepsilon, z) = \sum _{n=0} ^{\infty} x_{n}(z) \varepsilon ^{kn}$.
After substituting the formal series $\hat{x}(\varepsilon, z)$ to \eqref{eq:3} we get
\begin{equation*}
c_{q} \sum _{n=0}^{\infty} x_{n}^{(q)} (z) \varepsilon ^{k(n+q)} + \dots +c_{1}\displaystyle \sum _{n=0} ^{\infty} x'_{n} (z) \varepsilon ^{k(n+1)} + c_{0} \displaystyle \sum _{n=0} ^{\infty} x_{n} (z) \varepsilon^{kn} = f(z).
\end{equation*}

After renumbering, we obtain
\begin{equation*}
c_{q} \sum _{n=q} ^{\infty} x_{n-q} ^{(q)} (z) \varepsilon ^{kn} +\dots+ c_{1}\sum _{n=1} ^{\infty} x'_{n-1} (z) \varepsilon ^{kn}  + c_{0} \sum _{n=0} ^{\infty} x_{n} (z) \varepsilon ^{kn} = f(z).
\end{equation*}

Comparing coefficients of respective powers of $\varepsilon$ we get the dependences
\begin{equation*}
x_{0}(z) = \frac{f(z)}{c_{0}},
\end{equation*}
\begin{equation*}
x_{n}(z)=-\frac{1}{c_{0}}\left(c_{n}x_{0}^{(n)}(z) + c_{n-1}x_{1}^{(n-1)}(z) + \dots + c_{1}x_{n-1}'(z)\right),\ n=1,\dots,q, 
\end{equation*}
\begin{equation*}
x_{n}(z)=-\frac{1}{c_0}\Big(c_q x_{n-q}^{(q)}(z)+\dots+c_1x'_{n-1}(z)\Big)\quad\text{for}\quad n\geq q.
\end{equation*}

From the above equations we see that the sequence $(x_n(z))_{n\geq 0}$ has the form $x_n(z)=b_nf^{(n)}(z)$ for some sequence $(b_n)_{n\geq 0}$ of complex numbers. To find this sequence observe that we have
\begin{equation*}
c_{q}\partial _{z} ^{q} b_{n-q} \partial _{z} ^{n-q} f(z) + \dots + c_{1}\partial _{z} b_{n-1} \partial_{z} ^{n-1} f(z) + c_{0}b_{n} \partial _{z} ^{n}f(z) = \left\{
    \begin{array}{ll}
      f(z) \textrm{ for } n = 0\\
      0 \textrm{ for } n > 0.
    \end{array}
  \right.
\end{equation*}

Now we can write down the appropriate difference equation
\begin{equation}
\label{eq:2_2}
c_{q}b_{n-q} + \dots + c_{1}b_{n-1} + c_{0}b_{n} = 0.
\end{equation}
With the initial condition $c_{0}b_{0}=1$, we get $b_{0} = \frac{1}{c_{0}}$.

The characteristic polynomial of the difference equation \eqref{eq:2_2} has the form
\begin{equation*}
w (\lambda) = c_{0} \lambda ^{q} + c_{1} \lambda ^{q-1} + \dots  + c_{q} = \lambda ^{q}\Big(c_{0} +\frac{c_{1}}{\lambda} + \dots + \frac{c_{q}}{\lambda ^{q}}\Big) = \lambda ^{q} P\left(\frac{1}{\lambda}\right).
\end{equation*}
The polynomial $P(\zeta)$ can also be written as
\begin{equation*}
P(\zeta) = c_{q}\prod _{i=1} ^{l} (\zeta - a_{i}) ^{p_{i}}.
\end{equation*}

Then the characteristic polynomial is in the following form
\begin{equation*}
w(\lambda) = \lambda ^{q} c_{q} \displaystyle \prod _{i=1} ^{l} \left(\frac{1}{\lambda} - a_{i}\right)^{p_{i}} = c_{q} \displaystyle \prod _{i=1} ^{l} (1- \lambda a_{i}) ^{p_{i}}.
\end{equation*}
Hence the characteristic polynomial has $l$ roots
$\lambda_{i} = \frac{1}{a_{i}}$ of multiplicity $p_{i}$ for $i=1,\dots,l$. According to the theory of difference equations with constant coefficients
we get
\begin{equation}
\label{eq:b_n}
b_{n} = \sum _{i=1} ^{l}\sum _{j=0} ^{p_{i}-1} d_{i,j} n^{\underline{j}} \left(\frac{1}{a_{i}}\right)^{n}\quad \text{for some constants}\quad d_{i,j} \in \mathbb{C},
\end{equation}
where $n^{\underline{j}}:=n(n-1)\dots(n-j+1)$.

Therefore
\begin{equation*}
\hat{x}(\varepsilon, z) = \sum _{n=0} ^{\infty} b_{n} \partial _{z} ^{n} f(z) \varepsilon ^{kn}=
\sum _{n=0} ^{\infty}\left(\sum _{i=1} ^{l}\sum _{j=0} ^{p_{i}-1} d_{i,j} n^{\underline{j}} \left(\frac{1}{a_{i}}\right)^{n}\right) \partial _{z} ^{n} f(z) \varepsilon ^{kn}
\end{equation*}

So, using the Taylor formula we can write the Borel transform of the solution $\hat{x}(\varepsilon, z)$ as
\begin{multline*}
\hat{\mathcal{B}} _{k} \hat{x} (\varepsilon,z) = \sum _{n=0} ^{\infty} \frac{ b_{n} \partial _{z} ^{n} f(z) \varepsilon ^{kn}}{\Gamma(1+\frac{kn}{k})} = \sum _{n=0} ^{\infty} \frac{ b_{n} \partial _{z} ^{n} f(z) \varepsilon ^{kn}}{n!}\\
= \sum _{i=1} ^{l}\sum _{j=0} ^{p_{i}-1} d_{i,j} \displaystyle \sum _{n=j} ^{\infty} \frac{n^{\underline{j} } \left(\frac{1}{a_{i}}\right)^{n} \partial _{z} ^{n} f(z) \varepsilon ^{kn}}{n!} = \sum _{i=1} ^{l}\sum _{j=0} ^{p_{i}-1} d_{i,j} \sum _{n=0} ^{\infty} \frac{\left(\frac{1}{a_{i}}\right)^{n+j} \partial_{z} ^{n+j} f(z) \varepsilon ^{k(n+j)}}{n!}\\
= \sum _{i=1} ^{l}\sum _{j=0} ^{p_{i}-1} d_{i,j} \left(\frac{\varepsilon^k }{a_{i}}\right)^{j} \sum _{n=0} ^{\infty} \frac{\partial _{z} ^{n} f^{(j)} (z) \left(\frac{\varepsilon ^{k}}{a_{i}}\right)^{n}}{n!} = \sum _{i=1} ^{l}\sum _{j=1} ^{p_{i}-1} \left(\frac{\varepsilon ^{k}}{a_{i}}\right)^{j}d_{i,j} f^{(j)} \left(z+\frac{\varepsilon ^{k}}{a_{i}}\right).
\end{multline*}

We can observe that if $f(z) \in \mathcal{O} ^{1}(\hat{S}_{kd-\arg a_{i}})$, then also $f^{(j)}(z) \in \mathcal{O} ^{1}(\hat{S}_{kd-\arg a_{i}})$ and the function $\varepsilon \mapsto f^{(j)}(z+\frac{\varepsilon ^{k}}{a_{i}})$ belongs to the space $\mathcal{O}^{k}(S_{d})$.

So the formal power series solution $\hat{x}(\varepsilon, z) = \sum _{n=0} ^{\infty} x_{n}(z) \varepsilon ^{kn}$ of \eqref{eq:3} is $k$-summable in the direction $d$.

$(\Longleftarrow)$ Let $\hat{x}(\varepsilon, z)$ be $k$-summable in the direction $d$. We can write a polynomial $P(\zeta)$ in a factored form: $P(\zeta) = (\zeta-a_{j}) \cdot \tilde{P}_{j}(\zeta)$, where 
$$
\tilde{P}_{j}(\zeta) = \frac{P(\zeta)}{(\zeta-a_{j})} = c_{q} (\zeta - a_{j})^{p_{j}-1}\prod _{i=1, i \neq j} ^{l} (\zeta - a_{i}) ^{p_{i}}\quad \text{for}\quad j = 1, 2, \dots , l.
$$
Let $y_{j} = y_{j}(\varepsilon, z) = \tilde{P}_{j}(\varepsilon ^{k} \frac{d}{dz}) \hat{x}(\varepsilon, z)$. We have $P(\varepsilon ^{k} \frac{d}{dz}) \hat{x}(\varepsilon, z)= f(z)$, so it means that $\left(\varepsilon ^{k} \frac{d}{dz} - a_{j}\right) y_{j} = f$.

By  Lemma \ref{le:1} and the definition of $y_{j}$ we know that $y_{j}$ is also $k$-summable in the direction $d$.

The solution of the equation $\left(\varepsilon ^{k} \frac{d}{dz} - a_{j}\right) y_{j} = f$ is $k$-summable for $j=1,2,\dots,l$, so according to Lemma \ref{le:2} the function $f$ can be continued analytically in some sectors in directions $kd-\arg a_{j}$ for $j=1,2,\dots,k$ with the exponential growth of order $1$. Thanks to the arbitrariness of the choice of $j$ we get a thesis.
\end{proof}  

\section{Summability of singularly perturbed moment differential equations}
We obtain similar results for moment differential equations. In this section we show corresponding lemmas and theorem that are generalisations of the ones presented in the previous section.

\begin{Lem}\label{le:3}
    Suppose $\hat{x}(\varepsilon, z) \in \mathcal{O}[[\varepsilon]]$ is $k$-summable in the direction $d$, $Q(\zeta)$ is a polynomial with complex coefficients and $(m(n))_{n\geq0}$ is a regular sequence of moments of order $\frac{1}{\tilde{k}}$. Then $Q(\varepsilon ^{k} \partial _{m,z})  \hat{x}(\varepsilon, z)$ is also $k$-summable in the direction $d$.
\end{Lem}
\begin{proof}
    We can 
    write a polynomial $Q(\zeta)$ in the form $Q(\zeta)= a_{n}\zeta ^{n} + a_{n-1} \zeta ^{n-1} + \dots +a_{1} \zeta +a_{0}$. The rest of the proof of this lemma results from the fact that the space of $k$-summable sequences forms a differential algebra, so it is stable under multiplication and derivation $\partial _{m,z}$ (see \cite[Propositions 10.13 and 10.21]{pas}).
\end{proof}

\begin{Lem}\label{le:4}
    Let $\hat{x}(\varepsilon, z) \in \EE[[\varepsilon]]$ be a formal solution of the equation 
    $$\left(\varepsilon \partial_{m,z} - a\right)\hat{x}(\varepsilon, z) = f(z),$$ where $a\in \CC \setminus \{0\}$, $\tilde{k}>0, d\in \RR$, $(m(n))_{n\geq0}$ be a regular sequence of moments of order $\frac{1}{\tilde{k}}$ and $f$ be a holomorphic function in a complex neighbourhood of zero. Then  a function $f$ can be continued analytically in some sector in the direction $d-\arg a$ 
    and its continuation is of exponential growth of order $\tilde{k}$   if and only if  $\hat{x}(\varepsilon, z)$ is $\tilde{k}$-summable in the direction $d$.
\end{Lem}

\begin{proof}
    Let $\hat{x}(\varepsilon, z)=\displaystyle \sum _{n=0} ^{\infty} \frac{x_{n}(z)}{m(n)}\varepsilon^{n}$. Then we can write the equation in the form: 
    $$\displaystyle \sum _{n=0} ^{\infty} \frac{\partial_{m,z} x_{n}(z)}{m(n)}\varepsilon^{n+1} - a\displaystyle \sum _{n=0} ^{\infty} \frac{x_{n}(z)}{m(n)}\varepsilon^{n} = f(z),$$
    $$\displaystyle \sum _{n=1} ^{\infty} \frac{\partial_{m,z}x_{n-1}(z)}{m(n-1)}\varepsilon^{n} - a\displaystyle \sum _{n=0} ^{\infty} \frac{x_{n}(z)}{m(n)}\varepsilon^{n} = f(z).$$
    So we obtain formulas: $x_{0}=-\frac{f(z)}{a}$, $x_{n}(z)=-\frac{1}{a^{n+1}}\partial^{n}_{m,z}f(z)m(n)$. Therefore we can write the solution of the equation in the form: $\hat{x}(\varepsilon, z)=-\frac{1}{a}\displaystyle \sum _{n=0} ^{\infty} \frac{1}{a^{n}}\partial^{n}_{m,z}f(z)\varepsilon^{n}$.

    $(\Longrightarrow)$
    Let \begin{equation} \label{borel1}
        y(\varepsilon, z)=\hat{\mathcal{B}}_{m,\varepsilon}\hat{x}(\varepsilon, z) = -\frac{1}{a} \displaystyle \sum _{n=0} ^{\infty} \frac{\partial_{m,z} ^{n} f(z)}{m(n)} \left(\frac{\varepsilon}{a}\right)^{n}.
    \end{equation}
    Using the Cauchy integral formula and the definition of moment functions for $z\in D_{\varepsilon}$, for any function $f$ we get the following result:
    $$\partial _{m,z} ^{n} f(0) = \frac{1}{2 \pi i} \displaystyle \oint _{|\omega|=\varepsilon} f(\omega) \displaystyle \int _{0} ^{\infty (\theta)} \frac{\zeta ^{n} e_{m} (\zeta \omega)}{\zeta \omega} d\zeta d\omega,$$
    so, using the Taylor formula we get: 
    \begin{equation}\label{taylor}
    \partial ^{n} _{m,z} f(z) = \frac{1}{2 \pi i} \displaystyle \oint _{|\omega|=\varepsilon}f(\omega) \displaystyle \int _{0} ^{\infty (\theta)} \zeta ^{n} E_{m} (z \zeta) \frac{e_{m} (\omega \zeta)}{\omega \zeta} d \zeta d \omega \end{equation} \\ where $ \theta \in (-\arg \omega - \frac{\pi}{2l}, - \arg \omega + \frac{\pi}{2l})$.
    
    The whole proof of this equation can be found in \cite[Proposition 3]{mi1}.

    Observe that $y(\varepsilon, z)$ satisfies:

    \begin{equation}\label{uklad1}
  \left\{
    \begin{array}{ll}
      (\partial _{m, \varepsilon} - \lambda \partial _{m, z})y=0\\
      y(0, z)=\tilde{f}(z) \textrm{, where} \tilde{f}(z) = -\frac{f(z)}{a}, \lambda = \frac{1}{a}
    \end{array}
  \right.
\end{equation}
    Using the result \eqref{taylor}, we conclude that:
    \begin{multline*}
    y(\varepsilon,z) = \displaystyle \sum _{n=0} ^{\infty} \frac{\lambda ^{n} \partial _{m,z} ^{n} \tilde{f} (z)}{m(n)}\varepsilon ^{n}\\ = \displaystyle \sum _{n=0} ^{\infty} \frac{1}{2 \pi i} \displaystyle \oint _{|\omega|=\varepsilon} \tilde{f}(\omega)  \displaystyle \int _{0} ^{\infty (\theta)} \frac{e _{m}(\zeta \omega)}{\zeta \omega} \zeta ^{n} E _{m}(\zeta z) d \zeta d \omega \cdot \frac{\lambda ^{n} \varepsilon ^{n}}{m(n)}\\= \frac{1}{2 \pi i} \displaystyle \oint _{|\omega|} \tilde{f}(\omega) \displaystyle \int _{0} ^{\infty (\theta)} \frac{e _{m}(\zeta \omega)}{\zeta \omega} E _{m}(\zeta z) E _{m}(\lambda \varepsilon \zeta) d \zeta.
    \end{multline*}
    We know that $\tilde{f}(z)\in \mathcal{O}^{\tilde{k}}(\hat{S}_{d+\arg \lambda})$. Deforming the path of integration with respect to $\omega$ as in \cite[Lemma 5]{mi1}, we get $y(\varepsilon, z) \in \mathcal{O}^{\tilde{k}}(\hat{S_{d}}\times D)$.

   $(\Longleftarrow)$ Let us consider the following equation: 
    \begin{equation}\label{uklad2}
  \left\{
    \begin{array}{ll}
      (\partial _{m, z} - \frac{1}{\lambda} \partial _{m, \varepsilon})y=0\\
      y(\varepsilon, 0)=\tilde{g}(\varepsilon).
    \end{array}
  \right.
\end{equation}

We can make an observation, that $(\partial _{m, z} - \frac{1}{\lambda} \partial _{m, \varepsilon})y=0$ is satisfied by $y(\varepsilon,z)$ if and only if $(\partial _{m, \varepsilon} - \lambda \partial _{m, z})y=0$ is satisfied by $y(\varepsilon,z)$.

    Thanks to the symmetry of (\ref{uklad1}) and (\ref{uklad2}), in this part of a proof we should apply the same reasoning.
    This proves that if $y(\varepsilon,z) \in \mathcal{O}^{\tilde{k}}(\hat{S_{d}}\times D)$ and $\tilde{g}(\varepsilon) = y(\varepsilon,0) \in \mathcal{O}^{\tilde{k}}(\hat{S}_{d})$, then as previously we conclude that $y(\varepsilon,z)\in\Oo^{\tilde{k}}(D\times\hat{S}_{d+\arg\lambda})$. Therefore $f(z)=-ay(0,z) \in \mathcal{O} ^{\tilde{k}} (\hat{S} _{d - \arg a})$.    
\end{proof}

\begin{Rem}
    We can observe that there holds the following equality 
    \begin{equation}\label{borel}
    \hat{\mathcal{B}}_{m,\varepsilon}\hat{x}(\varepsilon, 0)=-\frac{1}{a}\displaystyle \sum _{n=0} ^{\infty} \frac{\partial^{n}_{m,z}f(0)}{m(n)}\left(\frac{\varepsilon}{a}\right)^{n}=-\frac{1}{a}f\left(\frac{\varepsilon}{a}\right).
    \end{equation}
    So since $\hat{x}(\varepsilon, z)\in \mathcal{O}\{\varepsilon\}_{\tilde{k},d}$, we see that $ \hat{\mathcal{B}}_{m,\varepsilon}\hat{x}(\varepsilon, 0) \in \mathcal{O}^{\tilde{k}}(\hat{S}_{d}) $ and by (\ref{borel}) we get $f \in \mathcal{O} ^{\tilde{k}}(S_{d-\arg a})$.
\end{Rem}

\begin{Lem}\label{le:5}
    Let $\hat{x}(\varepsilon, z) \in \EE[[\varepsilon]]$ be a formal solution of the equation 
    $$\left(\varepsilon ^{k} \partial_{m,z} - a\right)\hat{x}(\varepsilon, z) = f(z),
    $$
    where $a\in \CC \setminus \{0\}$, $k\in \NN$ $\tilde{k}>0, d\in \RR$, $(m(n))_{n\geq0}$ is a regular sequence of moments of order $\frac{1}{\tilde{k}}$ and $f$ is a function holomorphic in a complex neighbourhood of zero. Then a function $f$ can be continued analytically in some sector in the direction $kd-\arg a$ 
    and its continuation is of exponential growth of order $\tilde{k}$   if and only if  $\hat{x}(\varepsilon, z)$ is $k\tilde{k}$-summable in the direction $d$.
\end{Lem}

\begin{proof}
    Repeating the proof of Lemma \ref{le:4} with $\varepsilon$ replaced by $\varepsilon ^{k}$, we see that the formal solution $\hat
    {x}(\varepsilon, z)$ is given by 
    $$\hat{x} (\varepsilon, z) = - \frac{1}{a} \displaystyle \sum _{n=0} ^{\infty} \frac{1}{a^{n}} \partial _{m,z} ^{n} f(z) \varepsilon ^{nk}.
    $$
    Let $\tilde{m}(n) := m \left( \frac{n}{k}\right)$ and 
    \begin{equation} \label{eq:form_y}
    \tilde{y}(\varepsilon, z):= \hat{\Bo} _{\tilde{m},\varepsilon} \hat{x}(\varepsilon, z) = -\frac{1}{a} \displaystyle \sum _{n=0} ^{\infty} \frac{\partial _{m, z} ^{n} f(z)}{m(n)} \left(\frac{\varepsilon ^{k}}{a}\right)^{n}.
    \end{equation}
    Then $\tilde{y}(\varepsilon, z) = y(\varepsilon ^{k}, z)$, where $y(\varepsilon, z)$ is defined by (\ref{borel1}).
    
    $(\Longrightarrow)$ If $f \in \mathcal{O} ^{\tilde{k}} (\hat{S}_{kd - \arg a})$, then by the proof of Lemma \ref{le:4} we conclude that $y(\varepsilon, z) \in \mathcal{O} ^{\tilde{k}} (\hat{S}_{kd} \times D)$. It means that $\tilde{y}(\varepsilon, z) \in \mathcal{O} ^{k \tilde{k}} (\hat{S}_{d} \times D)$.
    
    $(\Longleftarrow)$ If $\tilde{y}(\varepsilon, z) \in \mathcal{O} ^{k \tilde{k}} (\hat{S}_{d} \times D)$, then $y(\varepsilon, z) \in \mathcal{O} ^{\tilde{k}} (\hat{S}_{kd} \times D)$ and by the proof of Lemma \ref{le:4} we conclude that $f(z) \in \mathcal{O} ^{\tilde{k}} (\hat{S}_{kd-\arg a})$.
\end{proof}

\begin{Rem}\label{re:5}
    If $f(z) \in \mathcal{O}^{\tilde{k}}(\hat{S}_{kd-\arg a})$, then by Lemma \ref{le:5} we get $\tilde{y}(\varepsilon, z) = \hat{\Bo}_{\tilde{m}, \varepsilon} \hat{x}(\varepsilon, z) \in \mathcal{O}^{k \tilde{k}} (\hat{S}_{d} \times D)$. Moreover, by the Cauchy integral formula if $\tilde{y}(\varepsilon, z) \in \mathcal{O}^{K} (\hat{S}_{d} \times D)$, then $\varepsilon \partial _{\varepsilon} \tilde{y}(\varepsilon, z) \in \mathcal{O}^{K} (\hat{S}_{d} \times D)$ for every $K>0$. Hence, we conclude that if $f(z)\in \mathcal{O} ^{k \tilde{k}} (\hat{S}_{kd- \arg a})$, then $(\varepsilon \partial _{\varepsilon})^{j} \tilde{y}(\varepsilon, z) \in \mathcal{O}^{k \tilde{k}} (\hat{S}_{d} \times D)$ for every $j \in \NN_0$.
\end{Rem}


\begin{Thm}\label{thm:2}
Let $k\in\NN, \tilde{k}>0, d\in\RR$ and $(m(n))_{n\geq0}$ be a regular sequence of moments of order $\frac{1}{\tilde{k}}$. W assume that $\hat{x}(\varepsilon, z) \in \EE[[\varepsilon]]$ is a formal solution of the equation
 \begin{equation*}
 \label{eq:3a}
 P\left(\varepsilon^k \partial _{m, z}\right)x(\varepsilon, z)=f(z),
 \end{equation*}
 where $P(\zeta)$ is any polynomial of degree $q$ with complex coefficients, such that $P(0) \neq 0$. We also assume that  $a_{1}, \dots , a_{l}\in\CC$ are the roots of the equation $P(\zeta)=0$  with multiplicity respectively  $p_{1}, \dots , p_{l}$, where $\sum _{i=1} ^{l} p_{i} = q$, $f(z)$ is holomorphic in a complex neighbourhood of the origin and $d\in\RR$.
 
 Then $f$ can be continued analytically in some sectors in directions $kd-\arg a_{i}$ for $i=1, \dots ,l$ and its continuation is of exponential growth of order $\tilde{k}$ if and only if $\hat{x}(\varepsilon, z)$ is $k \tilde{k}$-summable in the direction $d$. 
\end{Thm}
\begin{proof}
    The proof of this theorem is similar to the proof of Theorem \ref{th1}.
    
    $(\Longrightarrow)$ If $b_{n}$ is given by (\ref{eq:b_n}), one can find coefficients $\tilde{d}_{i,j}$ which depend on $d_{i,j}$ such that 
    $$
    b_{n} = \displaystyle \sum _{i=1} ^{l} \displaystyle \sum _{j=0} ^{p_{i}-1} \tilde{d}_{i,j} n^{j} \left(\frac{1}{a_{i}}\right)^{n}.
    $$
    Then we can observe that
    $$\hat{x}(\varepsilon, z) = \displaystyle \sum _{n=0} ^{\infty} \left( \displaystyle \sum _{i=1} ^{l} \displaystyle \sum _{j=0} ^{p_{i} -1} d_{i,j} n^{j} \left( \frac{1}{a_{i}}\right) ^{n} \right) \partial _{m,z} ^{n} f(z) \varepsilon ^{kn}.
    $$
    Therefore $\tilde{m}$-Borel transform of $\hat{x}(\varepsilon, z)$ for $\tilde{m}(n):=m(\frac{n}{k})$ can be written in the following form:
    \begin{equation}\label{eq:bo_tilde}
    \hat{\Bo}_{\tilde{m},\varepsilon} \hat{x}(\varepsilon, z) = \displaystyle \sum _{n=0} ^{\infty} \frac{1}{\tilde{m}(kn)}\displaystyle \sum _{i=1} ^{l} \displaystyle \sum _{j=0} ^{p_{i} -1} d_{i,j} n^{j} \left( \frac{1}{a_{i}}\right) ^{n} \partial _{m,z} ^{n} f(z) \varepsilon ^{kn}.
    \end{equation}
    Since $m=(m(n))_{n\in\NN_0}$ has order $\tilde{k}$, we see that $\tilde{m}=(\tilde{m}(n))_{n\in\NN_0}=(m(\frac{n}{k}))_{n\in\NN_0}$ has order $k\tilde{k}$.
     Moreover, as in the proof of Lemma \ref{le:5} (see formula \eqref{eq:form_y}) and by Remark \ref{re:5} we can observe that for any $i=1,\dots,l$ and $j=0,\dots,p_i-1$ the function
     \begin{equation*}
     (\varepsilon,z)\longmapsto \displaystyle \frac{d_{i, j}}{k^{j}} \left(\varepsilon \partial _{\varepsilon}\right)^{j} \displaystyle \sum _{n=0} ^{\infty} \frac{\left(\frac{\varepsilon ^{k}}{a_{i}}\right)^{n} \partial ^{n} _{m,z} f(z)}{m(n)}
     \end{equation*}
     belongs to the space $\Oo^{k\tilde{k}}(\hat{S}_d\times D)$.
     By \eqref{eq:bo_tilde} we conclude that $\hat{x}(\varepsilon, z)$ is $k\tilde{k}$-summable in the direction $d$.

     $(\Longleftarrow)$ This proof is analogous to the proof of Theorem~\ref{th1}, but instead of Lemma~\ref{le:1}, we use Lemma~\ref{le:3} and instead of Lemma~\ref{le:2}, we use Lemma~\ref{le:5}.
\end{proof}

\section{Relation to moment partial differential equations in two variables}
In this section we study more general singular perturbation problems and their generalisations to moment derivatives. Namely we consider singularly perturbed equations 
\begin{equation}
\label{eq:singular}
    P\left(\varepsilon, z, \frac{d}{dz}\right) \hat{x}(\varepsilon,z) = \hat{f}(\varepsilon,z),
\end{equation}  
where $P(\xi,z,\zeta)\in\EE[\xi,\zeta]$ is a polynomial of two variables $(\xi,\zeta)$ with coefficients in the Banach space $\EE$ dependent on $z$, and $\hat{f}(\varepsilon,z)=\sum_{n=0}^{\infty}f_n(z)\varepsilon^n\in\EE[[\varepsilon]]$.

We also consider the moment versions of the equations (\ref{eq:singular}), where the differentiation $\frac{d}{dz}$ is replaced by the moment differentiation $\partial_{m,z}$, where $m=(m(n))_{n\geq 0}$ and $m(u)$ is a function of moments. 

We find the connection between the Borel transform of the formal solution of the singular perturbation problem and the formal solution of the initial problem for the appropriate partial differential equation in two complex variables.

To this end we extend the definition of Borel transform $\hat{\Bo}_{1,\varepsilon}: \mathbb{E}[[\varepsilon]] \rightarrow \mathbb{E}[[\varepsilon]]$ 
to the space of formal Laurent series $\mathbb{E}[[\varepsilon, \varepsilon ^{-1}]]$ as follows:
\begin{Def}
    The operator $\hat{\bar{\Bo}} _{1, \varepsilon}: \mathbb{E}[[\varepsilon, \varepsilon ^{-1}]] \rightarrow \mathbb{E}[[\varepsilon]]$ defined by 
\begin{equation}\label{eq:B_ext}
    \hat{\bar{\Bo}} _{1, \varepsilon} \Big( \sum _{n= -\infty} ^{+ \infty} u_{n} \varepsilon ^{n}\Big) := \sum _{n=0} ^{\infty} \frac{u_{n}}{n!} \varepsilon ^{n}
\end{equation}
is called the \emph{extended Borel transform}.
\end{Def}

\begin{Rem}\label{re:1}
   Directly by the definition we conclude that the extended Borel transform $\hat{\bar{\Bo}} _{1, \varepsilon}$ has the following properties:
    \begin{enumerate}
        \item $\hat{\bar{\Bo}} _{1, \varepsilon}$ is a linear operator on $\mathbb{E}[[\varepsilon, \varepsilon ^{-1}]]$, i.e.
        \begin{equation*}
        \hat{\bar{\Bo}} _{1, \varepsilon}(\alpha u_{1} + \beta u_{2}) = \alpha \hat{\bar{\Bo}} _{1, \varepsilon} u_{1} + \beta \hat{\bar{\Bo}} _{1, \varepsilon} u_{2}
        \ \text{for every}\ u_1,u_2\in\EE[[\varepsilon,\varepsilon^{-1}]]\ \text{and}\ \alpha,\beta\in\CC.
        \end{equation*}
        \item $\hat{\bar{\Bo}} _{1, \varepsilon}$ is an extension of the Borel transform $\hat{{\Bo}} _{1, \varepsilon}$, i.e. 
        \begin{equation*}
        \hat{\bar{\Bo}} _{1, \varepsilon} (u) = \hat{\Bo} _{1, \varepsilon} (u) \quad\text{for every}\quad u \in \mathbb{E}[[\varepsilon]].
        \end{equation*}
    \end{enumerate}
\end{Rem}

Using the operator $\hat{\bar{\Bo}} _{1, \varepsilon}$ we show the following important property of the Borel transform 
\begin{Lem}\label{le:6}
    For every $\hat{u} \in \mathbb{E}[[\varepsilon, \varepsilon ^{-1}]]$ holds
\begin{equation}\label{eq:prop_B}
\hat{\bar{\Bo}} _{1, \varepsilon} \Big(\varepsilon ^{-1} \hat{u}(\varepsilon)\Big) = \partial _{\varepsilon} \Big(\hat{\bar{\Bo}} _{1, \varepsilon}  \hat{u}(\varepsilon)\Big).
\end{equation}
\end{Lem}
\begin{proof}
    By the linearity of $\hat{\bar{\Bo}} _{1, \varepsilon}$ and of $\partial _{\varepsilon}$ it is sufficient to show that \eqref{eq:prop_B} holds for $\hat{u}(\varepsilon)= u_{n} \varepsilon ^{n}$ for any fixed $n \in \mathbb{Z}$. We consider 3 cases:
    \begin{itemize}
    \item If $n < 0$ then by \eqref{eq:B_ext} we see that $\hat{\bar{\Bo}} _{1, \varepsilon} (u_{n} \varepsilon ^{n-1}) = 0$ and analogously\\ $\partial _{\varepsilon} \Big(\hat{\bar{\Bo}} _{1, \varepsilon} (u_{n} \varepsilon ^{n})\Big) = \partial _{\varepsilon} 0 = 0$.
    \item If $n=0$ then  similarly by \eqref{eq:B_ext} we get $\hat{\bar{\Bo}} _{1, \varepsilon} (u_{0} \varepsilon ^{-1}) = 0$ and $\partial _{\varepsilon} \Big(\hat{\bar{\Bo}} _{1, \varepsilon} (u_{0} \varepsilon ^{0})\Big) = \partial _{\varepsilon} u_{0} = 0$.
    \item If $n>0$ then we have (see also \cite[Lemma 5.1]{lamisu})
    \begin{equation*}
    \hat{\bar{\Bo}} _{1, \varepsilon} (u_{n} \varepsilon ^{n-1}) = \hat{\Bo} _{1, \varepsilon} (u_{n} \varepsilon ^{n-1}) = \frac{u_{n}}{(n-1)!} \varepsilon ^{n-1}
    \end{equation*}
    and
    \begin{equation*}
    \partial _{\varepsilon} \Big(\hat{\bar{\Bo}} _{1, \varepsilon} (u_{n} \varepsilon ^{n})\Big) = \partial _{\varepsilon} \Big(\hat{\Bo} _{1, \varepsilon} (u_{n} \varepsilon ^{n})\Big) = \partial _{\varepsilon} \Big(\frac{u_{n}}{n!} \varepsilon ^{n}\Big) = \frac{u_{n}}{(n-1)!} \varepsilon ^{n-1}.
    \end{equation*}
\end{itemize}
\end{proof}

\begin{Rem}
The above fundamental property of the Borel transform is widely used in the context of WKB analysis (see for example \cite[Proposition 2.6]{ta}), singular perturbations (see for example \cite[Lemma 5.1]{lamisu}) 
or singular partial differential equations (see for example \cite[page 221]{hi}). 
\end{Rem}

Applying Lemma \ref{le:6} $k$ times we get
\begin{Prop}\label{prop:1}
  For every $\hat{u} \in \mathbb{E}[[\varepsilon, \varepsilon ^{-1}]]$ and every $k\in\NN$ we have
\begin{equation*}
\hat{\bar{\Bo}} _{1, \varepsilon} \Big(\varepsilon ^{-k} \hat{u}(\varepsilon)\Big) = \partial _{\varepsilon}^k \Big(\hat{\bar{\Bo}} _{1, \varepsilon}  \hat{u}(\varepsilon)\Big).
\end{equation*}
\end{Prop}
\medskip\par
Analogously, using the classical Borel transform $\hat{\Bo}_{1,\varepsilon}$ instead of the extended one $\hat{\bar{\Bo}} _{1, \varepsilon}$ we receive
\begin{Prop}\label{prop:2}
For every $\hat{u} \in \mathbb{E}[[\varepsilon]]$ and every $k\in\NN$ we conclude that
\begin{equation*}
\hat{\Bo} _{1, \varepsilon} \Big(\varepsilon^k \hat{u}(\varepsilon)\Big) = \partial _{\varepsilon}^{-k} \Big(\hat{\Bo} _{1, \varepsilon}  \hat{u}(\varepsilon)\Big),
\end{equation*}
where $\partial _{\varepsilon}^{-1}$ denotes antiderivative, i.e. $\partial _{\varepsilon}^{-1}\varphi(\varepsilon)=\int_0^{\varepsilon}\varphi(t)\,dt$.
\end{Prop}
\medskip\par
We apply Lemma \ref{le:6} to show that $\hat{x}(z,\varepsilon)$ is a formal power series solution of the singular perturbation problem  if and only if its Borel transform $\hat{\Bo}_{1,\varepsilon}\hat{x}(z,\varepsilon)$ is a formal power series solution of the appropriate Cauchy problem. Namely, we have 
\begin{Thm}\label{thm:3}
    We assume that
    \begin{equation}\label{eq:P}
    P(\xi,z,\zeta) = P_{0} (z,\zeta) + P_{1} (z,\zeta) \xi + \dots + P_{p} (z,\zeta) \xi^{p}\quad\text{for}\quad z\in D,\quad \xi,\zeta\in\CC
    \end{equation}
    is a polynomial of two variables $(\xi,\zeta)$ with coefficients in $\EE$ dependent of $z$ and of order $p$ with respect to $\xi$, and let $\hat{f} (\varepsilon,z) = \sum _{n=0} ^{\infty} f_{n}(z) \varepsilon ^{n} \in \mathbb{E}[[\varepsilon]]$. 
    
    Then $\hat{x} (\varepsilon,z) = \sum _{n=0} ^{\infty} x_{n}(z) \varepsilon ^{n} \in \mathbb{E}[[\varepsilon]]$ is a formal power series solution of a singularly perturbed equation
\begin{equation}\label{eq:sing_pert}
    P\left(\varepsilon,z, \frac{d}{dz}\right) \hat{x}(\varepsilon,z) = \hat{f}(\varepsilon,z)
\end{equation}    
    if and only if its Borel transform $\hat{y}(\varepsilon,z) := \hat{\Bo} _{1, \varepsilon} \hat{x} (\varepsilon,z)$ is a formal solution of the Cauchy problem
    \begin{equation}\label{eq:partial}
 \left \{
        \begin{array}{l}
              Q (\partial _{\varepsilon},z, \partial _{z}) \hat{y} (\varepsilon,y)    = \hat{F} (\varepsilon,y)\\
             \partial _{\varepsilon} ^{n} \hat{y} (0, z)   = x_{n} (z)\quad \text{for}\quad n = 0, 1, \dots, p-1,
        \end{array}
       \right. 
\end{equation}
where $Q(\chi,z,\zeta):=\chi^pP(\chi^{-1},z,\zeta)$ (i.e.
$Q (\partial_{\varepsilon},z,\partial_z)= \partial_{\varepsilon}^{p} P(\partial_{\varepsilon}^{-1},z,\partial_z)$), 
$\hat{F} (\varepsilon,z):= \partial _{\varepsilon} ^{p} \hat{\Bo} _{1, \varepsilon} \hat{f} (\varepsilon,z)$ 
    and the functions $x_{0} (z), \dots, x_{p-1}(z)$ satisfy the equations:
\begin{equation*}
    P_n\left(z,\frac{d}{dz}\right)x_0(z)+P_{n-1}\left(z,\frac{d}{dz}\right)x_1(z)+\dots+P_0\left(z,\frac{d}{dz}\right)x_n(z)=f_n(z)
\end{equation*}
for $n=0,\dots,p-1$.
\end{Thm}

\begin{proof}
    $(\Longrightarrow)$ Let $\hat{x}(\varepsilon,z)$ satisfies \eqref{eq:sing_pert}. It means that $\varepsilon ^{-p}P\left(\varepsilon,z, \frac{d}{dz}\right) \hat{x}(\varepsilon,z)=\varepsilon ^{-p} \hat{f}(\varepsilon,z)$ belongs to the space $\mathbb{E}[[\varepsilon, \varepsilon ^{-1}]]$. Hence 
    $$Q\left(\frac{1}{\varepsilon},z, \frac{d}{dz}\right) \hat{x}(\varepsilon,z) = \varepsilon ^{-p} \hat{f}(\varepsilon,z).
    $$
    Applying the extended Borel transform $\hat{\bar{\Bo}}_{1,\varepsilon}$ to both sides of the above equality and using Lemma \ref{le:6} we conclude that 
    $$
    Q\left(\partial _{\varepsilon},z, \partial _{z}\right) \hat{\bar{\Bo}}_{1,\varepsilon} \hat{x}(\varepsilon,z) = \partial _{\varepsilon} ^{p} \hat{\bar{\Bo}}_{1,\varepsilon} \hat{f}(\varepsilon,z).
    $$
    
    Since $\hat{x}(\varepsilon,z)$ and $\hat{f}(\varepsilon,z)$ belong to $\mathbb{E}[[\varepsilon]]$, by Remark \ref{re:1} we get 
    $$Q(\partial _{\varepsilon},z, \partial _{z}) \hat{y}(\varepsilon,z) = \hat{F}(\varepsilon,z),$$
    where $\hat{y}(\varepsilon,z) = \hat{\Bo}_{1,\varepsilon} \hat{x}(\varepsilon,z) = \sum _{n=0} ^{\infty} \frac{x_{n}(z)}{n!}\varepsilon ^{n}$ and $\hat{F}(\varepsilon,z) = \partial _{\varepsilon} ^{p} \hat{\Bo} _{1,\varepsilon} \hat{f}(\varepsilon,z)$.
    
    Moreover by the Taylor formula $$\hat{y}(\varepsilon,z) = \sum _{n=0} ^{\infty} \frac{\partial _{\varepsilon} ^{n} \hat{y}(0,z)}{n!}\varepsilon ^{n},$$ hence we get the initial conditions 
    \begin{equation*}
    \partial _{\varepsilon} ^{n} \hat{y}(0,z) = x_{n}(z)\quad \text{for}\quad n=0,1, \dots, p-1.
    \end{equation*}
    
    $(\Longleftarrow)$ We have
    $$P\left(\varepsilon,z, \frac{d}{dz}\right) = P_{0}\left(z,\frac{d}{dz}\right) + P_{1}\left(z,\frac{d}{dz}\right)\varepsilon + \dots + P_{p}\left(z,\frac{d}{dz}\right)\varepsilon ^{p}.$$
Hence we may write
\begin{multline*}
P\left(\varepsilon,z, \frac{d}{dz}\right)\Big(\sum _{n=0} ^{\infty} x_{n}(z) \varepsilon ^{n}\Big) = P_{0}\left(z,\frac{d}{dz}\right)x_{0}(z)\\ + \Big(P_{0}\left(z,\frac{d}{dz}\right)x_{1}(z) + P_{1}\left(z,\frac{d}{dz}\right)x_{0}(z)\Big)\varepsilon +\dots\\
+ \Big(P_{0}\left(z,\frac{d}{dz}\right)x_{n}(z) + P_{1}\left(z,\frac{d}{dz}\right)x_{n-1}(z) + \dots + P_{p}\left(z,\frac{d}{dz}\right)x_{n-p}(z)\Big)\varepsilon ^{n} + \dots.
\end{multline*}

    It means that the sequence $(x_{n}(z)) _{n \geq 0}$ of coefficients for $\hat{x}(\varepsilon,z)=\sum_{n=0}^{\infty}x_n(z)\varepsilon^n$ being the solution of \eqref{eq:sing_pert} have to satisfy the following equations:
\begin{equation}\label{eq:5}
  P_{0} \left(z,\frac{d}{dz}\right) x_{n} (z) + P_{1} \left(z,\frac{d}{dz}\right) x_{n-1} (z) + \dots + P_{n} \left(z,\frac{d}{dz}\right) x_{0} (z) = f_{n}(z)\quad \text{for} \quad n \leq p,  
\end{equation}
\begin{equation}\label{eq:6}
  P_{0} \left(z,\frac{d}{dz}\right) x_{n} (z) + P_{1} \left(z,\frac{d}{dz}\right) x_{n-1} (z) + \dots + P_{p} \left(z,\frac{d}{dz}\right) x_{n-p} (z) = f_{n}(z)\quad \text{for} \quad n \geq p.  
\end{equation}

On the other hand, by \eqref{eq:partial} we see that 
$$Q(\partial _{\varepsilon},z, \partial _{z}) \hat{\bar{\Bo}}_{1,\varepsilon} \hat{x}(\varepsilon,z) = \partial _{\varepsilon} ^{p} \hat{\bar{\Bo}}_{1,\varepsilon} f(\varepsilon,z).
    $$

    By Lemma \ref{le:6} we get 
    $$\hat{\bar{\Bo}}_{1,\varepsilon} \Big(Q\left(\frac{1}{\varepsilon},z, \frac{d}{dz}\right) \hat{x}(\varepsilon,z)\Big) = \hat{\bar{\Bo}}_{1,\varepsilon}\Big(\varepsilon ^{-p} f(\varepsilon,z)\Big),$$ so also
    $$\hat{\bar{\Bo}}_{1,\varepsilon}\Big(\varepsilon ^{-p} P\left(\varepsilon,z, \frac{d}{dz}\right) \hat{x}(\varepsilon,z)\Big) = \hat{\bar{\Bo}}_{1,\varepsilon}\Big(\varepsilon ^{-p} f(\varepsilon,z)\Big).$$

    It means that for every $n \geq p$ we have 
    $$P_{0}\left(z,\frac{d}{dz}\right)x_{n}(z) + P_{1}\left(z,\frac{d}{dz}\right)x_{n-1}(z) + \dots + P_{p}\left(z,\frac{d}{dz}\right)x_{n-p}(z) = f_{n}(z),$$
    so \eqref{eq:6} holds. Moreover, by the conditions on $x_{0}(z), \dots, x_{p-1}(z)$ also \eqref{eq:5} holds, which completes the proof.
\end{proof}

Repeating the same proof with the ordinary derivative $\frac{d}{dz}$ replaced by the moment derivative $\partial_{m,z}$ for some sequence of moments $m=(m(n))_{n\geq 0}$ we get

\begin{Thm}\label{thm:3m}
    We assume that
    \begin{equation}\label{eq:Pm}
    P(\xi,z,\zeta) = P_{0} (z,\zeta) + P_{1} (z,\zeta) \xi + \dots + P_{p} (z,\zeta) \xi^{p}\quad\text{for}\quad z\in D,\quad \xi,\zeta\in\CC
    \end{equation}
    is a polynomial of two variables $(\xi,\zeta)$ with coefficients in $\EE$ dependent of $z$ and of order $p$ with respect to $\xi$, and let $\hat{f} (\varepsilon,z) = \sum _{n=0} ^{\infty} f_{n}(z) \varepsilon ^{n} \in \mathbb{E}[[\varepsilon]]$. 
    
    Then $\hat{x} (\varepsilon,z) = \sum _{n=0} ^{\infty} x_{n}(z) \varepsilon ^{n} \in \mathbb{E}[[\varepsilon]]$ is a formal power series solution of a singularly perturbed equation
\begin{equation}\label{eq:sing_pertm}
    P(\varepsilon,z, \partial_{m,z}) \hat{x}(\varepsilon,z) = \hat{f}(\varepsilon,z)
\end{equation}    
    if and only if its Borel transform $\hat{y}(\varepsilon,z) := \hat{\Bo} _{1, \varepsilon} \hat{x} (\varepsilon,z)$ is a formal solution of the Cauchy problem
    \begin{equation}\label{eq:partialm}
 \left \{
        \begin{array}{l}
              Q (\partial _{\varepsilon}, z, \partial _{m,z}) \hat{y} (\varepsilon,z)    = \hat{F} (\varepsilon,z)\\
             \partial _{\varepsilon} ^{n} \hat{y} (0,z)   = x_{n} (z)\quad \text{for}\quad n = 0, 1, \dots, p-1,
        \end{array}
       \right. 
\end{equation}
where $Q (\partial_{\varepsilon},z,\partial_{m,z}):= \partial_{\varepsilon}^{p} P(\partial_{\varepsilon}^{-1},z,\partial_{m,z})$, 
$\hat{F} (\varepsilon,z):= \partial_{\varepsilon} ^{p} \hat{\Bo} _{1, \varepsilon} \hat{f} (\varepsilon,z)$ 
    and the functions $x_{0} (z), \dots, x_{p-1}(z)$ satisfy the equations:
\begin{equation*}
    P_n(z,\partial_{m,z})x_0(z)+P_{n-1}(z,\partial_{m,z})x_1(z)+\dots+P_0(z,\partial_{m,z})x_n(z)=f_n(z)
\end{equation*}
for $n=0,\dots,p-1$.
\end{Thm}

\begin{Rem}\label{re:3}
 Observe that a formal power series solution $\hat{x} (\varepsilon,z) = \sum _{n=0} ^{\infty} x_{n}(z) \varepsilon ^{n} \in \mathbb{E}[[\varepsilon]]$ of a singularly perturbed equation \eqref{eq:sing_pert} is uniquely determined for every inhomogeneity $\hat{f}(\varepsilon,z)\in\EE[[\varepsilon]]$ if and only if the operator $P_0(z,\frac{d}{dz})$ is a linear automorphism on the space $\EE$ (see \cite{lamisu2} for such type of conditions). The same remark holds in the moment case, when $\frac{d}{dz}$ is replaced by $\partial_{m,z}$.
\end{Rem}

\section{Moment pseudodifferential operators and formal solutions of singularly perturbed moment equations}
In this sections we apply Theorem \ref{thm:3m} and the previous results of the second author \cite{mi1,mi2} to study summable solutions of a singularly perturbed equation \eqref{eq:sing_pert}.

In the further study we will assume that the function $P(\xi,z,\zeta)$ given by (\ref{eq:P}) satisfies the condition
\begin{equation*}
 \partial_{z} P(\xi,z,\zeta)\equiv 0\quad\text{for every}\quad z\in D,\quad \xi,\zeta\in\CC,
\end{equation*}
which means that $P(\xi,z,\zeta)$ does not depend on $z$, so 
\begin{equation*}
P(\xi,z,\zeta)=P(\xi,0,\zeta)=P_{0} (\zeta) + P_{1} (\zeta) \xi + \dots + P_{p} (\zeta) \xi^{p}=:P(\xi,\zeta)\quad\text{for}\quad \xi,\zeta\in\CC,\ z\in D
\end{equation*}
is a polynomial in two variables $(\xi,\zeta)$ with constant coefficients.

We also assume that $P(0,\zeta)=P_0(\zeta)\not\equiv 0$ and $\partial_{\zeta}P(0,\zeta)=0$. It means that $P(0,\zeta)=\text{const.}=P_0\neq 0$.

For such polynomials $P(\xi,\zeta)$ we consider
the following singular problem for moment differential equation
\begin{equation*}
 P(\varepsilon,\partial_{m,z})x(\varepsilon,z)=f(\varepsilon,z),
\end{equation*}
where $m=(m(n))_{n\geq 0}$ is a sequence of moments of given order $s>0$.

By $\Gamma_s=(\Gamma_s(n))_{n\in\NN_0}$ we denote the sequence $(\Gamma(1+sn))_{n\in\NN_0}$ for given $s>0$.

Since for every such sequence $m$ of order $s>0$ there exist constants $A,B<\infty$ such that
\begin{equation*}
 A^n\leq \frac{m(n)}{\Gamma_s(n)}\leq B^n\quad\text{for every}\quad n\in\NN_0,
\end{equation*}
one can show that for every $k>0$ and $d\in\RR$ holds the equivalence 
$$
\hat{x}(\varepsilon,z)\in\EE\{\varepsilon\}_{k,d}\quad \text{if and only if}\quad \hat{\Bo}_{m/\Gamma_s,z}\hat{x}(\varepsilon,z)\in\EE\{\varepsilon\}_{k,d}.
$$
Moreover $\hat{x}(\varepsilon,z)$ is a formal solution of the perturbation problem 
$$P(\varepsilon,\partial_{\Gamma_s,z})x(\varepsilon,z)=f(\varepsilon,z)$$ if and only if $\hat{y}(\varepsilon,z)=\hat{\Bo}_{m/\Gamma_s,z}\hat{x}(\varepsilon,z)$ is a formal solution of the perturbation problem 
$$
P(\varepsilon,\partial_{m,z})y(\varepsilon,z)=\hat{\Bo}_{m/\Gamma_s,z}f(\varepsilon,z).
$$
For this reason it is sufficient to consider the case when $m=\Gamma_s$ for some $s>0$. Note that in the special case $s=1$ we get $m=\Gamma_1=(n!)_{n\geq 0}$ and the moment derivative $\partial_{m,z}$ becomes the ordinary derivative $\frac{d}{dz}$. So our study also covers 
the perturbation problem for ordinary differentiation
$$P\left(\varepsilon,\frac{d}{dz}\right)x(\varepsilon,z)=f(\varepsilon,z).$$

By the fundamental theorem of algebra we may factorise the polynomial $P(\xi,\zeta)$ as
\begin{equation}\label{eq:factorise}
P(\xi,\zeta)=P_0\prod_{j=1}^{n_0}\prod_{l=1}^{m_j}\Big(1-\xi\lambda_{jl}(\zeta)\Big)^{r_{jl}},
\end{equation}
where
\begin{equation*}
\sum_{j=1}^{n_0} \sum_{l=1}^{m_j} r_{jl}=p.
\end{equation*}
Here $\lambda_{jl}(\zeta)$ are the roots of the characteristic equation $Q(\lambda,\zeta)=0$ of multiplicity $r_{jl}$ for $l=1,\dots,m_j$ and $j=1,\dots,n_0$, where
\begin{equation*}
 Q(\lambda,\zeta):=\lambda^pP\left(\frac{1}{\lambda},\zeta\right)
 =P_0\lambda^p+P_1(\zeta)\lambda^{p-1}+\dots+P_p(\zeta).
\end{equation*}

Observe that every $\lambda_{jl}(\zeta)$ is an algebraic function. It means by the implicit function theorem that every function $\lambda_{jl}(\zeta)$ is holomorphic on $\CC$ except at a finite number of singular or branching points, and this function has a moderate growth at infinity. More precisely, there exist a \emph{pole order} $q_j\in\QQ$ and a \emph{leading term} $a_{jl}\in\CC\setminus\{0\}$ such that
\begin{equation*}
\lim_{\zeta\to\infty}\frac{\lambda_{jl}(\zeta)}{\zeta^{q_j}}=a_{jl}\quad\text{for}\quad l=1,\dots,m_j\quad\text{and}\quad j=1,\dots,n_0.
\end{equation*}
We denote it shortly by $\lambda_{jl}(\zeta)\sim a_{jl} \zeta^{q_j}$.

Generally, if $\lambda(\zeta)$ is an algebraic function such that $\lambda(\zeta)\sim a\zeta^q$ for some $a\in\CC\setminus\{0\}$ and $q\in\QQ$, then there exists $r_0<\infty$ and $\kappa\in\NN$ such that $\lambda(\zeta)$ is a holomorphic function of the variable $\omega=\zeta^{1/\kappa}$ for $|\zeta|>r_0$ with a pole of some order $n$ at infinity. It means that the function
$\omega\mapsto \lambda(\omega^{\kappa})$ has the Laurent series expansion 
\begin{equation*}
\lambda(\omega^{\kappa})=\sum_{j=-n}^{\infty}\frac{b_j}{\omega^j}
\end{equation*}
at infinity for some coefficients $b_j\in\CC$ with $b_{-n}=a$ and $n=q\kappa\in\ZZ$. This expansion is convergent for $|\omega|>r_0^{1/\kappa}$ with a pole of order $n$ at infinity.

To avoid the ramification point at $z=0$, for $P$ satisfying (\ref{eq:factorise}) we take $\kappa\in\NN$ satisfying the condition
\begin{multline}\label{eq:kappa}
   \kappa=\inf \Big\{n\in\NN\colon \lambda_{jl}(\zeta^n)\in\Oo\big(\{\zeta\in\CC\colon|\zeta|>r_0\}\big)\quad \text{and this function has a pole}\\ \text{at infinity for every}\quad l=1,\dots,m_j, j=1,\dots,n_0\quad \text{and for some}\quad r_0>0\Big\}.
\end{multline}
In other words, $\kappa$ is the smallest natural number such that every function $\lambda_{jl}(\zeta)$ ($l=1,\dots,m_j$, $j=1,\dots,n_0$) given by (\ref{eq:factorise}) is a holomorphic function of the variable $\omega=\zeta^{1/\kappa}$ for $|\zeta|>r_0$. 

We can use the following version of \cite[Lemma 3]{mi1} for such chosen $\kappa$
\begin{Prop}\label{prop:rami}
 Let $P(\xi,\zeta)$ be a polynomial for two variables, $\kappa\in\NN$, $s>0$ and $f(\varepsilon,z)\in\EE[[\varepsilon]]$. Then $\hat{u}(\varepsilon,z)$ is a formal solution of
$$
P(\varepsilon,\partial_{\Gamma_s,z})u(\varepsilon,z)=f(\varepsilon,z)
$$
 if and only if $\hat{v}(\varepsilon,z):=\hat{u}(\varepsilon,z^{\kappa})$ is a formal solution of
 $$
 P(\varepsilon,\partial_{\Gamma_{s/\kappa},z}^{\kappa})v(\varepsilon,z)=f(\varepsilon,z^{\kappa}).
 $$
\end{Prop}
\begin{proof}
    To prove this equivalence it is sufficient to observe that 
    $$
    (\partial_{\Gamma_s,z}u)(\varepsilon,z^{\kappa})=\partial_{\Gamma_{s/\kappa},z}^{\kappa}(u(\varepsilon,z^{\kappa}))=\partial_{\Gamma_{s/\kappa},z}^{\kappa}v(\varepsilon,z).
    $$
\end{proof}

Since the polynomial $\tilde{P}(\xi,\zeta):=P(\xi,\zeta^{\kappa})$ has the same factorisation (\ref{eq:factorise}) as $P(\xi,\zeta)$ with $\lambda_{jl}(\zeta)$ replaced by $\tilde{\lambda}_{jl}(\zeta):=\lambda_{jl}(\zeta^{\kappa})$, we see that every function $\tilde{\lambda}_{jl}(\zeta)$ is holomorphic for $|\zeta|>r_0$ and it satisfies $\tilde{\lambda}_{jl}(\zeta)\sim a_{jl}\zeta^{q_j\kappa}$ with $q_j\kappa\in\ZZ$.

For such functions we may define the following pseudodifferential operators

\begin{Def}[see also {\cite[Definition 13]{mi2}}]
\label{def:pseudo}
Let $s>0$, $k=1/s$ and $\lambda(\zeta)$ be a holomorphic function for $|\zeta|\geq r_0$ and of moderate growth at infinity. A \emph{moment pseudodifferential operator} $\lambda(\partial_{\Gamma_{s},z})\colon \Oo(D)\to\Oo(D)$ is defined by
\begin{gather}
\label{eq:lambda_1}
\lambda(\partial_{\Gamma_{s},z})\varphi(z):=\frac{1}{2\pi i}\oint_{|w|=\varepsilon}\varphi(w)\int_{r_0}^{e^{i\theta}\infty}\lambda(\zeta)\mathbf{E}_{s}(z\zeta)k(w\zeta)^{k-1}e^{-(\zeta w)^{k}}\,d\zeta dw
\end{gather}
for every $\varphi(z)\in\Oo(D_r)$ and $|z|<\varepsilon<r$, where $\theta\in(-\arg w -\frac{\pi}{2k}, -\arg w +\frac{\pi}{2k})$, $\oint_{|w|=\varepsilon}\,dw$ means that we integrate along the positively oriented circle of radius $\varepsilon$ and $\mathbf{E}_{s}(z)=\sum_{n=0}^{\infty}\frac{z^n}{\Gamma(1+ns)}$ denotes the Mittag-Leffler function of index $s$.
\end{Def}

By Proposition \ref{prop:rami} we can extend the definition
of moment pseudodifferential operators from $\kappa=1$ to the case when $\kappa>1$, where $\kappa$ is given by (\ref{eq:kappa}). Namely we have 

\begin{Def}
\label{def:pseudo_2}
Let $s>0$, $\kappa\in\NN$ and $\lambda(\zeta)$ be a holomorphic function of the variable $\omega=\zeta^{1/\kappa}$ for $|\zeta|\geq r_0$ and of moderate growth at infinity. A \emph{moment pseudodifferential operator} $\lambda(\partial_{\Gamma_{s},z})\colon \Oo_{1/\kappa}(D)\to\Oo(D_{1/\kappa})$ is defined by
\begin{gather}
\label{eq:lambda_2}
\lambda(\partial_{\Gamma_{s},z})\varphi(z):=\lambda(\partial_{\Gamma_{s/\kappa},w}^{\kappa})(\varphi(w^{\kappa}))\Big|_{w=z^{1/\kappa}}
\end{gather}
for every $\varphi(z)\in\Oo_{1/\kappa}(D)$.
\end{Def}

We consider a formal power series solution $\hat{x}(\varepsilon,z)=\sum_{n=0}^{\infty}x_n(z)\varepsilon^n\in\EE[[\varepsilon]]$ of the perturbation problem
\begin{equation}\label{eq:Problem}
    P(\varepsilon, \partial_{\Gamma_s,z}) \hat{x}(\varepsilon,z) = \hat{f}(\varepsilon,z)\in\EE[[\varepsilon]],
\end{equation}
where the polynomial $P(\xi,\zeta)$ is given by (\ref{eq:factorise}).

Since $P_0=\text{const.}\neq 0$,
by Theorem \ref{thm:3m} the perturbation problem (\ref{eq:Problem}) has a unique formal power series solution.

In the next theorem we will show how to decompose the formal solution of  the singular perturbation problem on the sum of formal power series connected with the appropriate pseudodifferential operators. 
\begin{Thm}\label{thm:4}
 Let $\hat{x}(\varepsilon,z)$ be a formal solution of the singular perturbation problem
 \begin{equation}\label{eq:tilde}
  P(\varepsilon,\partial_{\Gamma_s,z})\hat{x}(\varepsilon,z)=\hat{f}(\varepsilon,z),
 \end{equation}
where
$
 P(\varepsilon,\zeta)= P_0+P_1(\zeta)\varepsilon+\dots+P_p(\zeta)\varepsilon^p
 =P_0(1-\varepsilon\lambda_1(\zeta))^{m_1}\dots(1-\varepsilon\lambda_l(\zeta))^{m_l}$ and $\hat{f}(\varepsilon,z)\in\EE[[\varepsilon]]$.

Then we can decompose $\hat{x}(\varepsilon,z)$ as
$$\hat{x}(\varepsilon,z)=\sum_{\alpha=1}^l\sum_{\beta=1}^{m_{\alpha}}\hat{x}_{\alpha\beta}(\varepsilon,z),$$
where $\hat{x}_{\alpha\beta}(\varepsilon,z)$ is a formal power series satisfying
\begin{equation*}
 \hat{x}_{\alpha\beta}(\varepsilon,z)=q_{\alpha\beta}(\varepsilon, \partial_{\Gamma_s,z})\hat{f}(\varepsilon,z)
\end{equation*}
with
\begin{equation*}
 q_{\alpha\beta}(\varepsilon, \zeta)=c_{\alpha\beta}(\zeta)\sum_{n=0}^{\infty}\frac{n!}{(n+1-\beta)!}\lambda_{\alpha}^n(\zeta)\varepsilon^n.
\end{equation*}
Here $c_{\alpha\beta}(\zeta)$ are algebraic functions satisfying
\begin{equation*}
 q_n(\zeta)=\sum_{\alpha=1}^l\sum_{\beta=1}^{m_{\alpha}}c_{\alpha\beta}(\zeta)\frac{n!}{(n+1-\beta)!}\lambda_{\alpha}^n(\zeta)\quad\text{for}\quad n\in\NN_0
\end{equation*}
for the sequence of polynomials $(q_n(\zeta))_{n\geq 0}$ being the unique solution of the difference equation
\begin{equation}\label{eq:diffe}
 P_0 q_{n}(\zeta)+P_1(\zeta)q_{n-1}(\zeta)+\dots+P_p(\zeta)q_{n-p}(\zeta)=0\quad\text{for}\quad n\geq 1
\end{equation}
with the initial conditions $q_0(\zeta)\equiv 1/P_0$ and $q_{-1}(\zeta)\equiv\dots\equiv q_{1-p}(\zeta)\equiv 0$.

Moreover, if $\hat{f}(\varepsilon,z)\in\EE[[\varepsilon]]_{\tilde{s}}$ for some $\tilde{s}\geq 0$ then $\hat{x}_{\alpha\beta}(\varepsilon,z)\in\EE[[\varepsilon]]_{\max\{\tilde{s},q_{\alpha}s\}}$, where $q_{\alpha}$ is a pole order of $\lambda_{\alpha}(\zeta)$ at the infinity.
\end{Thm}
\begin{proof}
Since $P_0\neq 0$, by Theorem \ref{thm:3m} (see also Remark \ref{re:3}) the equation (\ref{eq:tilde})
has the unique formal power series solution $\hat{x}(\varepsilon,z)$.
We will show that this formal power series solution is given by
\begin{equation}\label{eq:for_x}
 \hat{x}(\varepsilon,z)=\sum_{n=0}^{\infty}\varepsilon^n q_n(\partial_{\Gamma_s,z})\hat{f}(\varepsilon,z),
\end{equation}
where the sequence of polynomials $(q_n(\zeta))_{n\geq 0}$ satisfies the difference equation (\ref{eq:diffe})
with the initial conditions $q_0(\zeta)\equiv 1/P_0$ and $q_{-1}(\zeta)\equiv\dots\equiv q_{1-p}(\zeta)\equiv 0$.

Indeed, putting \eqref{eq:for_x} into \eqref{eq:tilde} we get
\begin{multline*}
 P(\varepsilon,\partial_{\Gamma_s,z})\left(\sum_{n=0}^{\infty}\varepsilon^n q_n(\partial_{\Gamma_s,z})\hat{f}(\varepsilon,z)\right)=\hat{f}(\varepsilon,z)+\sum_{n=1}^{\infty}\varepsilon^n P_0 q_n(\partial_{\Gamma_s,z})\hat{f}(\varepsilon,z)\\
 +
 \sum_{n=0}^{\infty}\varepsilon^{n+1} P_1(\partial_{\Gamma_s,z}) q_n(\partial_{\Gamma_s,z})\hat{f}(\varepsilon,z)
 +\dots+
 \sum_{n=0}^{\infty}\varepsilon^{n+p} P_p(\partial_{\Gamma_s,z}) q_n(\partial_{\Gamma_s,z})\hat{f}(\varepsilon,z)\\
 = \hat{f}(z,\varepsilon) + \sum_{n=1}^{\infty}\varepsilon^n \Big(P_0 q_n(\partial_{\Gamma_s,z})+P_1(\partial_{\Gamma_s,z}) q_{n-1}(\partial_{\Gamma_s,z})+\dots+P_p(\partial_{\Gamma_s,z}) q_{n-p}(\partial_{\Gamma_s,z})\Big)\hat{f}(\varepsilon,z)\\
  = \hat{f}(\varepsilon,z).
\end{multline*}
Since the sequence $(q_n(\zeta))_{n\geq 0}$ satisfies the difference equation \eqref{eq:diffe}, by the general theory of difference equations we conclude that
\begin{equation*}
 q_n(\zeta)=\sum_{\alpha=1}^l\sum_{\beta=1}^{m_{\alpha}}c_{\alpha\beta}(\zeta)\frac{n!}{(n+1-\beta)!}\lambda_{\alpha}^n(\zeta)\quad\text{for}\quad n\in\NN_0
\end{equation*}
for some algebraic functions $c_{\alpha\beta}(\zeta)$.
It means that $\hat{x}(\varepsilon,z)=\sum_{\alpha=1}^l\sum_{\beta=1}^{m_{\alpha}}\hat{x}_{\alpha\beta}(\varepsilon,z)$,
where
\begin{equation*}
 \hat{x}_{\alpha\beta}(\varepsilon,z)=q_{\alpha\beta}(\varepsilon, \partial_{\Gamma_s,z})\hat{f}(\varepsilon,z)
\end{equation*}
and
\begin{equation*}
 q_{\alpha\beta}(\varepsilon, \zeta)=c_{\alpha\beta}(\zeta)\sum_{n=0}^{\infty}\frac{n!}{(n+1-\beta)!}\lambda_{\alpha}^n(\zeta)\varepsilon^n.
\end{equation*}

Let $\hat{f}(\varepsilon,z)=\sum_{n=0}^{\infty}f_n(z)\varepsilon^n\in\EE[[\varepsilon]]_{\tilde{s}}$ for some $\tilde{s}\geq 0$.
To calculate the Gevrey order of $\hat{x}_{\alpha\beta}(\varepsilon,z)=\sum_{n=0}^{\infty}x_n(z)\varepsilon^n$, first observe that since the function $c_{\alpha\beta}(\zeta)$ has a moderate growth at infinity, it does not change the Gevrey order of $\hat{x}_{\alpha\beta}(\varepsilon,z)$. Hence without loss of generality we may assume that $c_{\alpha\beta}(\zeta)\equiv 1$. In this case using \cite[Lemma 1]{mi1} we can estimate $x_n(z)$ by
\begin{equation*}
 \|x_n(z)\|_{\EE}\leq\sum_{k=0}^{n}\frac{k!}{(k+1-\beta)!}\|\lambda_{\alpha}^k(\partial_{\Gamma_s,z})f_{n-k}(z)\|_{\EE}\leq AB^k\sum_{k=0}^n\Gamma(1+q_{\alpha}sk)\Gamma(1+\tilde{s}(n-k))
\end{equation*}
for some $A,B<\infty$. It means that there exist $C,D<\infty$ such that
\begin{equation*}
 \|x_n(z)\|_{\EE}\leq C D^n\Gamma(1+\max\{\tilde{s},q_{\alpha}s\}n)\quad\text{for every}\quad n\in\NN_0.
\end{equation*}
\end{proof}

\section{Multisummable solutions of singularly perturbed equations}
In this section we consider the general case of the perturbation problem
\begin{equation}\label{eq:perturb}
    P(\varepsilon, \partial_{\Gamma_s,z}) \hat{x}(\varepsilon,z) = \hat{f}(\varepsilon,z)\in\EE[[\varepsilon]],
\end{equation}
where 
 \begin{equation*}
 P(\xi,\zeta)=P_0 \tilde{P}(\xi,\zeta)=P_0\prod_{j=1}^{n_0}\prod_{l=1}^{m_j}\Big(1-\xi\lambda_{jl}(\zeta)\Big)^{r_{jl}}
 \end{equation*}
with $\lambda_{jl}(\zeta)\sim a_{jl}\zeta^{q_j}$ and $q_j=\mu_j/\nu_j$ for some relatively prime $\mu_j,\nu_j\in\ZZ$, where $l=1,\dots,m_j$ and $j=1,\dots,n_0$. 

Additionally we assume that $q_1>q_2>\dots>q_{n_0}$ and
\begin{equation*}
\tilde{n}:=
\left\{
   \begin{array}{lll}
 0&\textrm{for}& q_1\leq 0\\
\max\{i\colon q_i>0\}&\textrm{for}& q_1>0.
\end{array}
\right.
\end{equation*}

\bigskip\par
In this case we have the following sufficient conditions for summability and multisummability of formal solutions.
\begin{Thm}\label{thm:6}
 Assume that $P(\xi,\zeta)$ is a polynomial of two variables given by \eqref{eq:factorise}.
 Then the formal power series solution $\hat{x} (\varepsilon,z) = \sum _{n=0} ^{\infty} x_{n}(z) \varepsilon ^{n} \in \mathbb{E}[[\varepsilon]]$ of a singularly perturbed equation
\begin{equation*}
    P(\varepsilon, \partial_{\Gamma_s,z}) \hat{x}(\varepsilon,z) = \hat{f}(\varepsilon,z)\in\EE[[\varepsilon]]_{q_1s}
\end{equation*}
is of Gevrey order $q_1s$.

Moreover:
\begin{enumerate}
 \item[(a)] if $\tilde{n}=0$ then this formal power series
 $\hat{x} (z, \varepsilon)$ is convergent.
 \item[(b)] if $\tilde{n}=1$ and $\hat{f}(\varepsilon,z)$ satisfies additionally the condition
 \begin{equation}\label{eq:cond_b}
 \hat{\Bo}_{1/(q_1s),\varepsilon}\hat{f}(\varepsilon,z)\in\Oo^{1/(q_1s),1/s}\Big(\hat{S}_d\times\hat{S}_{(d+2k\pi+\arg a_{1l})/q_1}\Big)
\end{equation}
for $l=1,\dots,m_1$ and $k=0,\dots,\mu_1-1$,
 then $\hat{x}(\varepsilon,z)$ is $1/(q_1s)$-summable in a direction $d$.
\item[(c)] if $\tilde{n}\geq 2$ and $\hat{f}(\varepsilon,z)$ satisfies additionally the condition
\begin{equation}\label{eq:cond_c}
 \hat{\Bo}_{1/(q_js),\varepsilon}\hat{f}(\varepsilon,z)\in\Oo^{1/(q_js),1/s}\Big(\hat{S}_d\times\hat{S}_{(d+2k\pi+\arg a_{jl})/q_j}\Big)
\end{equation}
for $l=1,\dots,m_j$, $k=0,\dots,\mu_j-1$ and $j=1,\dots,\tilde{n}$,
then $\hat{x} (\varepsilon,z)$ is $(1/(q_{\tilde{n}}s),\dots,1/(q_1s))$-multisummable in an admissible multidirection
\linebreak
$(d_{\tilde{n}},\dots,d_{1})\in\RR^{\tilde{n}}$.
\end{enumerate}
\end{Thm}
\begin{proof}
 By Theorem \ref{thm:3m} the formal power series $\hat{y}(\varepsilon,z):= \hat{\bar{\Bo}} _{1, \varepsilon} \hat{x} (\varepsilon,z)$ 
 is a formal solution of the Cauchy problem
    \begin{equation}\label{eq:CPr}
 \left \{
        \begin{array}{l}
              Q (\partial _{\varepsilon}, \partial_{\Gamma_s,z}) y (\varepsilon,z)    = \hat{F} (\varepsilon,z)\\
             \partial _{\varepsilon} ^{n} y (0,z)   = x_{n} (z)\quad \text{for}\quad n = 0, 1, \dots, p-1,
        \end{array}
       \right. 
\end{equation}
where $Q (\partial_{\varepsilon},\partial_{\Gamma_s,z}):= \partial_{\varepsilon}^{p} P(\partial_{\varepsilon}^{-1},\partial_{\Gamma_s,z})$, 
$\hat{F} (\varepsilon,z):= \partial _{\varepsilon} ^{p} \hat{\bar{\Bo}} _{1, \varepsilon} \hat{f} (\varepsilon,z)$ 
    and the functions $x_{0} (z), \dots, x_{p-1}(z)$ are determined inductively by $x_0(z)=\frac{g_0(z)}{P_0}$ and
\begin{equation*}
    x_n(z)=\frac{1}{P_0}\left(g_n(z)-P_n(\partial_{\Gamma_s,z})x_0(z)-P_{n-1}(\partial_{\Gamma_s,z})x_1(z)-\dots-P_1(\partial_{\Gamma_s,z})x_{n-1}(z)\right)
\end{equation*}
for $n=1,\dots,p-1$.

If we take the sequence of moments $m=(m(n))_{n\geq 0}:=\big(\frac{\Gamma(1+q_1sn)}{n!}\big)_{n\geq 0}$ and $m$-Borel transform we see that $\hat{\Bo}_{1/(q_1s),\varepsilon}\hat{x}(\varepsilon,z)=\hat{\Bo}_{m,\varepsilon}\hat{y}(\varepsilon,z)=:\hat{w}(\varepsilon,z)$. Hence applying $m$-Borel transform to the Cauchy problem \eqref{eq:CPr} and using \cite[Proposition 7]{mi2} we get that $\hat{w}(\varepsilon,z)$ is a formal solution of the following moment differential equation 
\begin{equation*}
 \left \{
        \begin{array}{l}
              Q (\partial _{\Gamma_{q_1s},\varepsilon}, \partial _{\Gamma_s,z}) w (\varepsilon,z)    = \hat{G}(\varepsilon,z)\\
             \partial _{\Gamma_{q_1s},\varepsilon} ^{n} w (0,z)   = x_{n}(z)\quad \text{for}\quad n = 0, 1, \dots, p-1,
        \end{array}
       \right. 
\end{equation*}
where $\hat{G}(\varepsilon,z):=\hat{\Bo}_{m,\varepsilon}\hat{F}(\varepsilon,z)=\partial _{\varepsilon} ^{p} \hat{\Bo} _{1/(q_1s), \varepsilon} \hat{f} (\varepsilon,z)$. 

If $\hat{f}(\varepsilon,z)\in\EE[[\varepsilon]]_{q_1s}$ then the formal power series $\hat{G}(\varepsilon,z)$ is convergent and its sum $G(\varepsilon,z)$ is holomorphic in a complex neighbourhood of the origin $D^2$. In this case by \cite[Theorem 1]{mi3} we conclude that $\hat{w} (\varepsilon,z)\in\EE[[\varepsilon]]_0$ and in the consequence $\hat{x}(\varepsilon,z)\in\EE[[\varepsilon]]_{q_1s}$.
\medskip\par
(a) If $\tilde{n}=0$ then $q_1\leq 0$ and we deduce that 
$\hat{x}(\varepsilon,z)\in\EE[[\varepsilon]]_{0}$. (It also follows directly from Theorem \ref{thm:4}.)
\medskip\par
(b) If $\tilde{n}=1$ and $\hat{f}(\varepsilon,z)$ satisfies additionally the condition
 \begin{equation*}
 \hat{\Bo}_{1/(q_1s),\varepsilon}\hat{f}(\varepsilon,z)\in\Oo^{1/(q_1s),1/s}\Big(\hat{S}_d\times\hat{S}_{(d+2k\pi+\arg a_{1l})/q_1}\Big)
\end{equation*}
for $l=1,\dots,m_1$ and $k=0,\dots,\mu_1-1$,
then the formal power series $\hat{G}(\varepsilon,z)$ is convergent and its sum $G(\varepsilon,z)$ belongs to the space
$\Oo^{1/(q_1s),1/s}\Big(\hat{S}_d\times\hat{S}_{(d+2k\pi+\arg a_{1l})/q_1}\Big)$ for $l=1,\dots,m_1\quad\text{and}\quad k=0,\dots,\mu_1-1$.

Then by \cite[Theorem 2]{mi3} also $\hat{w}(\varepsilon,z)=\hat{\Bo}_{1/(q_1s),\varepsilon}\hat{x}(\varepsilon,z)$ is convergent and its sum belongs to the same space as $G(\varepsilon,z)$. In consequence
$\hat{x}(\varepsilon,z)$ is $1/(q_1s)$-summable in a direction $d$ (see also \cite[Theorem 4]{mi3}).

\medskip\par
(c) If $\tilde{n}\geq 2$ then by Theorem \ref{thm:4}
we may decompose the formal solution $\hat{x}(\varepsilon,z)$ by
$$\hat{x}(\varepsilon,z)=\sum_{j=1}^{n_0}\sum_{l=1}^{m_j}\sum_{\beta=1}^{r_{jl}}\hat{x}_{jl\beta}(\varepsilon,z),$$
where $\hat{x}_{jl\beta}(\varepsilon,z)$ is a formal power series satisfying
\begin{equation*}
 \hat{x}_{jl\beta}(\varepsilon,z)=q_{il\beta}(\varepsilon, \partial_{\Gamma_s,z})\hat{f}(\varepsilon,z)
\end{equation*}
with
\begin{equation*}
 q_{jl\beta}(\varepsilon, \zeta)=c_{jl\beta}(\zeta)\sum_{n=0}^{\infty}\frac{n!}{(n+1-\beta)!}\lambda_{jl}^n(\zeta)\varepsilon^n.
\end{equation*}
Here $c_{jl\beta}(\zeta)$ are algebraic functions satisfying
\begin{equation*}
 q_n(\zeta)=\sum_{j=1}^{n_0}\sum_{l=1}^{m_j}\sum_{\beta=1}^{r_{jl}}c_{jl\beta}(\zeta)\frac{n!}{(n+1-\beta)!}\lambda_{jl}^n(\zeta)\quad\text{for}\quad n\in\NN_0
\end{equation*}
for the sequence of polynomials $(q_n(\zeta))_{n\geq 0}$ being the unique solution of the difference equation
\begin{equation*}
 P_0 q_{n}(\zeta)+P_1(\zeta)q_{n-1}(\zeta)+\dots+P_p(\zeta)q_{n-p}(\zeta)=0\quad\text{for}\quad n\geq 1
\end{equation*}
with the initial conditions $q_0(\zeta)\equiv 1/P_0$ and $q_{-1}(\zeta)\equiv\dots\equiv q_{1-p}(\zeta)\equiv 0$.

Observe that by Proposition \ref{prop:2}
\begin{equation*}
\hat{\Bo}_{1,\varepsilon}\hat{x}_{jl\beta}(\varepsilon,z)=q_{jl\beta}(\partial_{\varepsilon}^{-1},\partial_{\Gamma_s,z})\hat{\Bo}_{1,\varepsilon}\hat{f}(\varepsilon,z)\in\EE[[\varepsilon]]_{q_js-1}.
\end{equation*}
Hence formal power series
$$\hat{y}_{jl\beta}(\varepsilon,z):=\hat{\Bo}_{1,\varepsilon}\hat{x}_{jl\beta}(\varepsilon,z),\ j=1,\dots,n_0,\ l=1,\dots,m_j,\ \beta=1,\dots,r_{jl}$$
give the decomposition of the solution
\begin{equation*}
 \hat{y}(\varepsilon,z)=\sum_{j=1}^{n_0}\sum_{l=1}^{m_j}\sum_{\beta=1}^{r_{jl}}\hat{y}_{jl\beta}(\varepsilon,z)
\end{equation*}
of the Cauchy problem (\ref{eq:CPr}) as in \cite[Theorem 1]{mi3}.
If we take the sequence of moments $\tilde{m}=(\tilde{m}(n))_{n\geq 0}=\big(\frac{\Gamma(1+q_jsn)}{n!}\big)_{n\geq 0}$ and $\tilde{m}$-Borel transform we see that $$\hat{\Bo}_{1/(q_js),\varepsilon}\hat{x_{jl\beta}}(\varepsilon,z)=\hat{\Bo}_{\tilde{m},\varepsilon}\hat{y_{jl\beta}}(\varepsilon,z)=:\hat{w}_{jl\beta}(\varepsilon,z)\in\EE[[\varepsilon]]_0.$$
Moreover, by (\ref{eq:cond_c}) we get
$$
 \hat{\Bo}_{1/(q_js),\varepsilon}\hat{f}(\varepsilon,z)=
 \hat{\Bo}_{\tilde{m},\varepsilon}\hat{\Bo}_{1,\varepsilon}\hat{f}(\varepsilon,z)
 \in\Oo^{1/(q_js),s}\Big(\hat{S}_d\times\hat{S}_{(d+2k\pi+\arg a_{jl})/q_j}\Big)
$$
for $l=1,\dots,m_j$, $k=0,\dots,\mu_j-1$.

Hence, as in the previous step (b), using \cite[Theorem 2]{mi3} we conclude that $\hat{x}_{jl\beta}(z,\varepsilon)$ is $1/(q_js)$-summable in a direction $d$ for $j=1,\dots,\tilde{n}$, which means that $\hat{x} (z, \varepsilon)$ is $(1/(q_{\tilde{n}}s),\dots,1/(q_1s))$-multisummable in an admissible multidirection
\linebreak
$(d_{\tilde{n}},\dots,d_{1})\in\RR^{\tilde{n}}$.
\end{proof}

\end{document}